\documentclass[11pt]{article}
\usepackage{amsmath,amssymb,amsfonts,amsthm} 
\usepackage{mathtools}
\usepackage{stmaryrd} 
\usepackage{color}
\usepackage{calc}
\usepackage[version=3,arrows=pgf-filled]{mhchem}
\usepackage{csquotes}
\usepackage{colortbl}
\usepackage{extarrows}
\usepackage[ruled]{algorithm2e} 
\usepackage{mathrsfs} 
\usepackage{import} 
\usepackage{enumitem}
\usepackage{appendix}
\usepackage{fullpage}
\usepackage{pgfplots}
\usepackage{tikz}
\usetikzlibrary{patterns,decorations.markings}
\usepackage{thmtools}   

\tikzstyle{field} = [rectangle, draw = black, align = center]
\tikzstyle{arrow} = [thick,->,>=stealth]

\usepackage[hidelinks]{hyperref}

\newcommand{\R}{{\mathbb R}}
\newcommand{\N}{{\mathbb N}}
\newcommand{\Z}{{\mathbb Z}}
\newcommand{\C}{{\mathbb C}}

\newcommand{\curlN}{\mathcal{N}}

\newcommand{\curlO}{\mathcal{O}}

\newcommand{\curlV}{\mathcal{V}}

\newcommand{\curlU}{\mathcal{U}}

\newcommand{\curlP}{\mathcal{P}}
\newcommand{\curlX}{\mathcal{X}}
\newcommand{\curlY}{\mathcal{Y}}
\newcommand{\curlZ}{\mathcal{Z}}
\newcommand{\curlE}{\mathcal{E}}
\newcommand{\curlM}{\mathcal{M}}
\newcommand{\snorm}[2][]{\ensuremath{\left\vert#2\right\vert_{#1}}}
\newcommand{\norm}[2][]{\ensuremath{\left\|#2\right\|_{#1}}}
\newcommand{\scalarprod}[2]{\ensuremath{\left\langle#1,#2\right\rangle}}

\newcommand{\utw}{u_\text{tw}}
\newcommand{\vtw}{v_\text{tw}}
\newcommand{\cphase}{c_p}

\newcommand{\Ahet}{A_\text{het}}
\newcommand{\Bhet}{B_{1,\text{het}}}

\newtheorem{theorem}{Theorem}[section]
\newtheorem{lemma}[theorem]{Lemma}

\newtheorem{remark}[theorem]{Remark}

\makeatletter  
\def\@endtheorem{\qed\endtrivlist\@endpefalse } 
\makeatother

\newtheoremstyle{remarks}
{10pt} 
{10pt} 
{} 
{} 
{\bfseries} 
{.}
{ } 
{}

\theoremstyle{remarks}

\newtheorem*{rem*}{Remark}

\renewcommand{\Re}{\operatorname{Re}}
\renewcommand{\Im}{\operatorname{Im}}

\author{Bastian Hilder\footnote{Institut f\"ur Analysis, Dynamik und Modellierung, Universit\"at Stuttgart, Pfaffenwaldring 57, 70569 Stuttgart, Germany. E-mail adress: \href{mailto:bastian.hilder@mathematik.uni-stuttgart.de}{bastian.hilder@mathematik.uni-stuttgart.de}}}
\title{Modulating traveling fronts in a dispersive Swift-Hohenberg equation coupled to an additional conservation law}
\date{\today}

\begin{document}
\maketitle

\begin{abstract}
	We consider a one-dimensional Swift-Hohenberg equation coupled to a conservation law, where both equations contain additional dispersive terms breaking the reflection symmetry $x \mapsto -x$.
	This system exhibits a Turing instability and we study the dynamics close to the onset of this instability.
	First, we show that periodic traveling waves bifurcate from a homogeneous ground state.
	Second, fixing the bifurcation parameter close to the onset of instability, we construct modulating traveling fronts, which capture the process of pattern-formation by modeling the transition from the homogeneous ground state to the periodic traveling wave through an invading front.
	The existence proof is based on center manifold reduction to a finite-dimensional system.
	Here, the dimension of the center manifold depends on the relation between the spreading speed of the invading modulating front and the linear group velocities of the system.
	Due to the broken reflection symmetry, the coefficients in the reduced equation are genuinely complex.
	Therefore, the main challenge is the construction of persistent heteroclinic connections on the center manifold, which correspond to modulating traveling fronts in the full system.
\end{abstract}

\section{Introduction}

Pattern-forming systems admitting a conservation law structure are omnipresent in physical problems, see \cite{crossHohenberg93,matthewsCox00}.
In particular, they appear generically in hydrodynamical stability problems with a free boundary such as the B\'enard-Marangoni problem \cite{zimmermann14} or the flow down an inclined surface \cite{haeckerSchneiderZimmermann11}.
Here, the mean fluid hight formally acts as a conserved quantity.
Typically, these systems possess a homogeneous background state, which is destabilized when a bifurcation parameter increases beyond a critical value and a spatially periodic traveling wave, also called a wave train or pattern, bifurcates.
For fixed bifurcation parameter beyond the critical value, these periodic patterns often arise in the wake of an invading heteroclinic front, which connects the unstable ground state to the spatially periodic traveling wave, see e.g.\ \cite{deeLanger83,benJacobBrandDeeKramerLanger85,vanSaarloos03}.

We model this invasion process by constructing so-called modulating traveling fronts, see \cite{colletEckmann86,eckmannWayne91,haragusSchneider99,fayeHolzer15,hilder20} for examples.
Modulating traveling fronts are solutions of the form $u(t,x) = \curlU(x-ct,x-\cphase t)$, where $\curlU(\xi,p)$ is periodic with respect to its second argument and satisfies the asymptotic boundary conditions
\begin{align*}
	\lim_{\xi \rightarrow -\infty} \curlU(\xi,p) = \utw(p) \text{ and } \lim_{\xi \rightarrow +\infty} \curlU(\xi,p) = u_\text{ground}.
\end{align*}
Here, $\utw$ denotes the bifurcating periodic traveling wave, $u_\text{ground}$ denotes the ground state, $c$ is the spreading speed of the modulating front and $\cphase$ is the phase velocity of $\utw$, see Figure \ref{fig:modfront}.
We point out that solutions of this type, i.e.~solutions modeling the invasion of a pattern into an unstable ground state, are also referred to as pattern-forming fronts, see e.g.~\cite{ebertSpruijtVanSaarloos04,gohDeRijk20arXiv}.

\begin{figure}
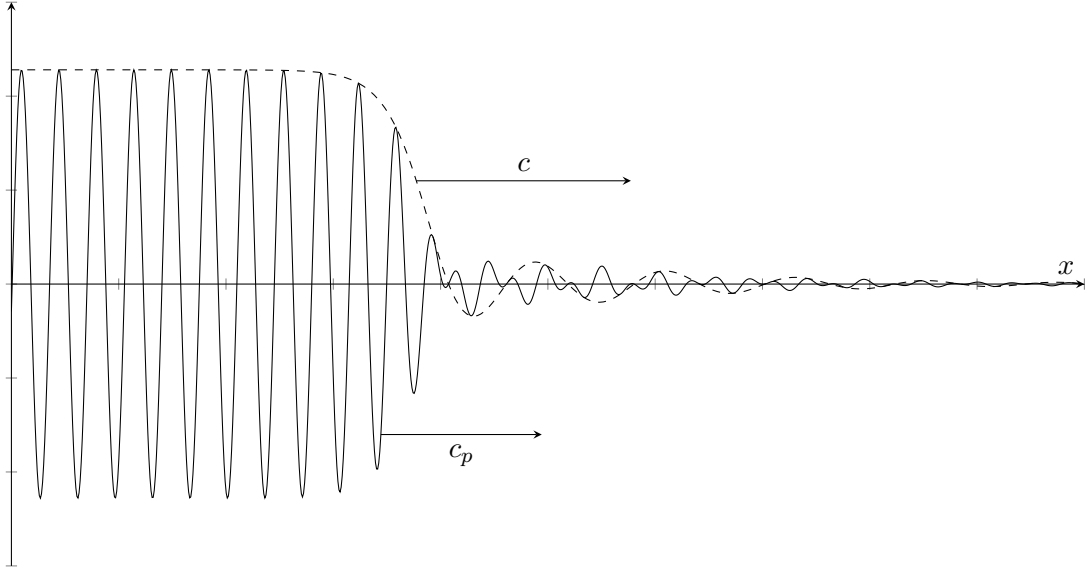

	\centering
	\include{Images/SchematicPic/modfrontDispersionSchematic}
	\caption{Schematic depiction of a modulating traveling front with spreading speed $c$, which connects to a periodic traveling wave with phase velocity $\cphase$. Note that, in general, the amplitude (dashed line) does not decay monotonically to zero, see Remark \ref{rem:nonMonotonicDecay}.}
	\label{fig:modfront}
\end{figure}

In \cite{hilder20}, modulating traveling fronts have been constructed for a Swift-Hohenberg equation, which is coupled to a conservation law, that is
\begin{subequations}
	\begin{align}
		\partial_t u &= -(1+\partial_x^2)^2 u + \varepsilon^2 \alpha_0 u + uv - u^3, \label{eq:SH}\\
		\partial_t v &= \partial_x^2 v + \gamma\partial_x^2(u^2), \label{eq:con}
	\end{align}
\end{subequations}
with $u(t,x), v(t,x) \in \R$, $x \in \R$, $t \geq 0$ and parameters $\gamma \in \R$, $\alpha_0 > 0$ and $0 < \varepsilon \ll 1$.
In this system, $\alpha_0$ is the bifurcation parameter and $\alpha_c = 0$ is the critical value at which the ground state $(u_\text{ground},v_\text{ground}) = (0,0)$ destabilizes.
Since $\varepsilon$ is small and $\alpha_0 > \alpha_c = 0$, the system is close to the onset of instability.

One particular feature of the system \eqref{eq:SH}--\eqref{eq:con} is its reflection symmetry, that is, it is invariant with respect to $x \mapsto -x$.
Notably, system \eqref{eq:SH}--\eqref{eq:con} shares this property for example with the B\'enard-Marangoni problem, which serves as motivation for the analysis performed in \cite{hilder20}.
It turns out that in general, breaking the reflection symmetry of a system has a profound impact on its dynamics and changes the structures that can be observed, see e.g.\ \cite{burkeHoughtonKnobloch09}.
In this paper, we study the effect of removing the reflection symmetry on the presence of modulating traveling fronts.
Therefore, we break the reflection symmetry of the system \eqref{eq:SH}--\eqref{eq:con} by including dispersive terms and study the following dispersive Swift-Hohenberg equation with an additional conservation law,
\begin{subequations}
	\label{eq:modGLDis}
	\begin{align}
		\partial_t u &= -(1+\partial_x^2)^2u + \varepsilon^2 \alpha_0 u + c_u \partial_x^3 u + uv + u\partial_x u - u^3, \label{eq:SHDis} \\
		\partial_t v &= \partial_x^2 v + c_v \partial_x v + \gamma_1 \partial_x^2(u^2) + \gamma_2\partial_x(u^2), \label{eq:ConDis}
	\end{align}
\end{subequations}
with $u(t,x), v(t,x) \in \R$, $x \in \R$, $t \geq 0$ and parameters $c_u,c_v,\gamma_1,\gamma_2 \in \R$ with $c_u \neq 0$, $\alpha_0 > 0$ and $0 < \varepsilon \ll 1$.
Again $\alpha_0$ is the bifurcation parameter and we study the system close to the onset of instability since $\alpha_0 > 0$ and $\varepsilon > 0$ is small.
System \eqref{eq:SHDis}--\eqref{eq:ConDis} is a prototypical model of a pattern-forming system admitting a conservation law, which has a broken reflection symmetry and it is motivated by real-world systems such as the flow down an inclined plane or a tilted B\'enard-Marangoni problem, that is a flow down a heated inclined plane.

It turns out that the system \eqref{eq:SHDis}--\eqref{eq:ConDis} exhibits a zoo of modulating traveling front solutions depending on the spreading speed of the front $c$, and the dispersion parameters $c_u$ and $c_v$.
In particular, we identify 5 different parameter scenarios, in which modulating fronts can be constructed, see Figure \ref{fig:Scenarios}.
These cover almost all possible choices for $c, c_u$ and $c_v$ and the only notable exclusion is the case that $c$ is close to $c_u$ with respect to $\varepsilon$, see also Remark \ref{rem:speedCloseToPhaseVel}.
For each scenario we derive a reduced system, which approximately governs the amplitude of the modulating front.
This reduced system and particularly its dimension is the main difference between the parameter regimes.
We refer to Section \ref{sec:IntroExistenceModFronts} for a detailed outline of the different scenarios and a discussion of the corresponding reduced systems.

To conclude the introduction of the model \eqref{eq:SHDis}--\eqref{eq:ConDis} we briefly discuss the choice of the additional dispersive terms.
On the linear level the symmetry is broken via $c_u \partial_x^3 u$ in \eqref{eq:SHDis} and $c_v \partial_x v$ in \eqref{eq:ConDis}.
In particular, the third derivative in the Swift-Hohenberg equation is a typical choice in the pure dispersive Swift-Hohenberg equation (i.e.~without an additional conservation law), see e.g.~\cite{burkeHoughtonKnobloch09,harizEtAl19}.
It turns out that the results of our analysis crucially rely the presence of this third derivative rather than a first derivative as in \eqref{eq:ConDis}, see Remark \ref{rem:speedCloseToPhaseVel}.
For similar reasons, we also require $c_u \neq 0$.
In contrast, $c_v = 0$ is an admissible choice and therefore, the analysis only requires dispersion in the Swift-Hohenberg part.
Furthermore, on the nonlinear level the reflection symmetry is broken via a Burgers nonlinearity $u \partial_x u$ in both \eqref{eq:SHDis} and \eqref{eq:ConDis}.
Such a term is frequently used to model nonlinear advection and in this paper they are additionally motivated mathematically, see Remark \ref{rem:nonlinDispersionMotivationS1}.

\subsection{Linear dispersion relation and traveling wave solutions}
We now examine the instability mechanism, which drives the formation of patterns in the system \eqref{eq:SHDis}--\eqref{eq:ConDis}, in more detail.
The linearization of \eqref{eq:SHDis}--\eqref{eq:ConDis} about the homogeneous ground state $(u,v) = (0,0)$ is solved by $u(t,x) = \exp(\lambda_u(k)t-ikx)$ and $v(t,x) = \exp(\lambda_v(k) t - ikx)$, where
\begin{subequations}
	\begin{align}
		\lambda_u(k) &= -(1-k^2)^2 + \varepsilon^2 \alpha_0 + c_u ik^3, \label{eq:lambdaU}\\
		\lambda_v(k) &= -k^2 - c_v ik. \label{eq:lambdaV}
	\end{align}
\end{subequations}
For $\alpha_0 > 0$, we find that $\Re(\lambda_u(k)) > 0$ for all Fourier wave number $k$ that are sufficiently close to the critical wave number $k_c = \pm 1$, see Figure \ref{fig:instabMechanism}.
Hence, the ground state is indeed unstable for $\alpha_0 > 0$.

\begin{figure}
	\centering
	\begin{tikzpicture}
		\begin{axis}[
				xmin = -2,
				xmax = 2,
				ymin = -2,
				ymax = 2,
				xtick = {-1,1},
				ytick = {0},
				xticklabels = {$-k_c$,$+k_c$},
				yticklabels = \empty,
				xlabel = $k$,
				x label style={anchor=above},
				axis x line = center,
				axis y line = center,
				legend pos = outer north east,
				legend style = {font=\small},
			]
			\addplot[mark=none,samples=100,domain=-2:2,color=red] {-(1-x^2)^2+0.5};
			\addlegendentry{$\Re(\lambda_u)$};
			\addplot[mark=none,samples=100,domain=-2:2,dashed,color=red] {+x^3};
			\addlegendentry{$\Im(\lambda_u)$};
			\addplot[mark=none,samples=100,domain=-2:2,color=blue] {-x^2};
			\addlegendentry{$\Re(\lambda_v)$};
			\addplot[mark=none,samples=100,domain=-2:2,color=blue,dashed] {-x};
			\addlegendentry{$\Im(\lambda_v)$};
		\end{axis}
	\end{tikzpicture}
	\caption{Plot of the real and imaginary parts of the spectral curves $\lambda_u$ and $\lambda_v$, see \eqref{eq:lambdaU} and \eqref{eq:lambdaV}.}
	\label{fig:instabMechanism}
\end{figure}
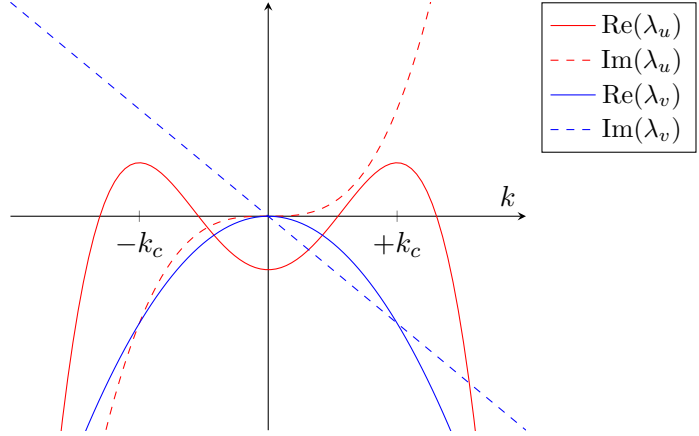

From this instability mechanism, we can predict the linear phase and group velocities, which play an important role in the construction of modulating traveling fronts.
The phase velocity describes the movement speed of a periodic traveling wave with a certain spatial frequency and the group velocity gives the speed of slowly varying modulations of the wave train or the speed of localized wave packages, see \cite{doelmanSandstedeScheelSchneider09,sandstedeScheelSchneiderUecker12} and in particular \cite[Figure 1.3]{sandstedeScheel04}.
Since the system \eqref{eq:SHDis}--\eqref{eq:ConDis} is not reflection symmetric (and thus, the spectral curves \eqref{eq:lambdaU} and \eqref{eq:lambdaV} have a non-trivial imaginary part) the bifurcating pattern is expected to have a non-zero phase velocity.
To obtain an intuition for the expected phase and group velocity we rewrite the solutions of the linearization about the ground state as
\begin{align*}
	u(t,x) &= \exp(\Re(\lambda_u)(k) t) \exp(i(\Im(\lambda_u(k)) t - kx)), \\
	v(t,x) &= \exp(\Re(\lambda_u)(k) t) \exp(i(\Im(\lambda_u(k)) t - kx)),
\end{align*}
where $\lambda_u, \lambda_v$ are given by \eqref{eq:lambdaU}--\eqref{eq:lambdaV}.
In particular, these solutions are $2\pi$-periodic in $\theta_u = \Im(\lambda_u(k)) t - kx$ and $\theta_v = \Re(\lambda_v(k)) t - kx$, respectively.
Therefore, we call 
\begin{align*}
	\omega_{\text{lin},u}(k) &:= \Im(\lambda_u(k)) = c_u k^3, \\
	\omega_{\text{lin},v}(k) &:= \Im(\lambda_v(k)) = -c_v k
\end{align*}
the linear dispersion relation of $u$ and $v$, respectively.
We then define the linear phase and group velocities corresponding to $u$ and $v$ by
\begin{align}
	c_\text{lin,phase,u}(k) &:= \dfrac{\omega_{\text{lin},u}(k)}{k} = c_u k^2, \nonumber\\
	c_\text{lin,group,u}(k) &:= \dfrac{d\omega_{\text{lin},u}(k)}{dk} = 3 c_u k^2, \label{eq:cGroupU}\\
	c_\text{lin,phase,v}(k) &= c_\text{lin,group,v}(k) = -c_v, \label{eq:cGroupV}
\end{align}
where the spatial wave number $k$ is close to the critical wave number $k_c = \pm 1$ for which the ground state is unstable, see Figure \ref{fig:instabMechanism}.
Here, it is worth pointing out that the group and phase velocity of the dispersive Swift-Hohenberg equation are different since $c_u \neq 0$ and the linear dispersive term in \eqref{eq:SHDis} is a third order derivative rather than a first order one.
It turns out that this is crucial to the analysis in this paper, see Remark \ref{rem:speedCloseToPhaseVel}.

Using these observations, we conjecture that when the homogeneous ground state destabilizes, a spatially periodic traveling wave bifurcates, which is of the form $(u,v)(t,x) = (\utw,\vtw)(\omega t - kx)$ and is $2\pi$-periodic.
Here, the spatial wave number $k$ is close to the critical wave number $k_c = 1$ and $\omega$ is the temporal frequency related to $k$ by the nonlinear dispersion relation $\omega_\text{nl}$, see e.g.~\cite{doelmanSandstedeScheelSchneider09,sandstedeScheelSchneiderUecker12}.
Since the instability is driven by the instability of the spectral curve $\lambda_u$ associated with the $u$-equation \eqref{eq:SHDis}, we expect that the phase velocity $\cphase(k) := \omega_\text{nl}(k)/k$ of the traveling wave is close to the linear phase velocity $c_\text{lin,phase,u}$ and thus satisfies
\begin{align*}
	\cphase(k) \approx c_\text{lin,phase,u}(k_c) = c_u,
\end{align*}
as long as $k$ is close to $k_c$.
This expectation can be made rigorous by using center manifold theory and we formally state the result, see Section \ref{sec:travelingWaves} and Lemma \ref{lem:travelingWaves} for details.

\begin{lemma}\label{lem:travelingWavesIntro}
	Let $\alpha_0 > 0$.
	Then, if $\varepsilon > 0$ is sufficiently small and $\gamma_1, \gamma_2 \in \R$ are close to zero, there exists a non-trivial traveling wave solution $(\utw,\vtw)$ of \eqref{eq:SHDis}--\eqref{eq:ConDis} with amplitude $\utw = \curlO(\varepsilon)$ and $\vtw = \curlO(\varepsilon^2)$.
	In particular, the traveling wave is stationary and periodic in the co-moving frame $p = x-\cphase t$ and its phase velocity is given by $\cphase = c_u + \curlO(\varepsilon^2)$.
\end{lemma}

\begin{remark}
	The statement of Lemma \ref{lem:travelingWavesIntro}, and also the remainder of this paper, is restricted to the spatial wave number $k = k_c = 1$. However, we expect that the results are also valid for $k$ close to $k_c$, since for every $\alpha_0$, an open interval of spatial wave numbers close to $k_c$ is destabilized, see Figure \ref{fig:instabMechanism}.
\end{remark}

\subsection{Existence of modulating traveling fronts: strategy, challenges and results}\label{sec:IntroExistenceModFronts}

We now discuss the strategy to establish the existence of modulating traveling fronts in the system \eqref{eq:SHDis}--\eqref{eq:ConDis} and arising challenges.
Recall that modulating traveling fronts are solutions $(u,v)(t,x) = (\curlU,\curlV)(\xi,p)$ with $\xi = x-ct$ and $p = x-\cphase t$, which satisfy
\begin{align}
	\lim_{\xi \rightarrow -\infty} (\curlU,\curlV)(\xi,p) = (\utw,\vtw)(p) \text{ and } \lim_{\xi \rightarrow +\infty} (\curlU,\curlV)(\xi,p) = (0,0).
	\label{eq:modfrontBC}
\end{align}
Here, $(\utw,\vtw)$ is the bifurcating traveling wave and $\cphase$ its phase velocity, which is given in Lemma \ref{lem:travelingWavesIntro}, and $c$ is the spreading speed of the modulating front.
Additionally, we assume that $c > 3c_u$ throughout the following discussion, i.e.~that the spreading speed of the front is faster than the group velocity associated to the periodic traveling wave, see also the following Remark \ref{rem:switchingPoint}.

\begin{remark}\label{rem:switchingPoint}
	Although we restrict the discussion to the existence of modulating traveling fronts, which spread faster than the group velocity of the periodic wave train (i.e.~$c > 3c_u$), it is possible to establish existence results for $c < 3c_u$.
	However, in this case we obtain ``reversed'' modulating traveling fronts, where the asymptotic states are interchanged, that is,
	\begin{align}
		\lim_{\xi \rightarrow -\infty} (\curlU,\curlV)(\xi,p) = (0,0) \text{ and } \lim_{\xi \rightarrow +\infty} (\curlU,\curlV)(\xi,p) = (\utw,\vtw)(p).
		\label{eq:revModfrontBC}
	\end{align}
	It is intuitively clear that the group velocity acts as a switching point since a modulating traveling front can be interpreted as a slowly varying modulation of the periodic traveling wave $(\utw,\vtw)$, which are expected to spread with the group velocity.
	We point out that this is similar to distinguishing the cases $c < 0$ and $c > 0$ in the reflection symmetric case.
\end{remark}

As in \cite{eckmannWayne91,haragusSchneider99,fayeHolzer15,hilder20}, the analysis is based on spatial dynamics and center manifold theory.
That is, we insert the modulating traveling front ansatz into \eqref{eq:SHDis}--\eqref{eq:ConDis} and rewrite the resulting system into a first order system with respect to the spatial variable $\xi = x-ct$, see \eqref{eq:spatDynFinal}.
In the non-dispersive case \eqref{eq:SH}--\eqref{eq:con} analyzed in \cite{hilder20} the main challenge is the construction of a sufficiently large center manifold since, at the bifurcation point, there are infinitely many eigenvalues on the imaginary axis.
It turns out that, unless the spreading speed of the modulating traveling front is close to the phase velocity of the bifurcating traveling wave (see Remark \ref{rem:speedCloseToPhaseVel}), the spectral situation is very different for the dispersive system \eqref{eq:SHDis}--\eqref{eq:ConDis}.
That is, at the bifurcation point only a finite number of central eigenvalues lie on the imaginary axis and there is an $\varepsilon$-independent spectral gap to the remaining  hyperbolic spectrum, see Lemmas \ref{lem:specLSH} and \ref{lem:specLcon}.
In particular, classical center manifold theory, see e.g.\ \cite{vanderbauwhedeIooss92,haragusIooss11}, is applicable in our setting to obtain a reduced system.

Thus, using center manifold theory, we reduce the infinite-dimensional spatial system to a finite-dimensional reduced system.
Then, the modulating traveling fronts in the full system correspond to heteroclinic orbits connecting the invading state to the origin in the reduced system.
It turns out that the construction of these heteroclinic connections is the main challenge in the dispersive case.
The difficulty is due to the presence of complex-valued coefficients in the reduced equation, which is caused by the dispersive terms.
In contrast the coefficients of the reduced equation in the reflection symmetric case are real-valued, which simplifies the analysis.

Analyzing the central spectrum (see Lemmas \ref{lem:specLSH} and \ref{lem:specLcon} for details), we find that the number of eigenvalues on the imaginary axis depends on the choice of the spreading speed $c$ of the modulating front in relation to the linear group velocities $c_\text{lin,group,u}(k_c) = 3c_u$ and $c_\text{lin,group,v}(k_c) = -c_v$, see \eqref{eq:cGroupU} and \eqref{eq:cGroupV}.
In particular, we distinguish the following scenarios, see also Figure \ref{fig:Scenarios}.

\begin{enumerate}[label=\textbf{Scenario \Roman*:},align=left,ref=\Roman*]
	\item $c\vert_{\varepsilon = 0} \neq -c_v$ and $c\vert_{\varepsilon = 0} \neq 3c_u$ \label{scenario1Intro}
	\item $c\vert_{\varepsilon = 0} \neq -c_v$ and $c = 3c_u + \varepsilon c_0$ with $c_0 \neq 0$ \label{scenario2Intro}
	\item $c = -c_v + \varepsilon^2 c_0$ with $c_0 \neq 0$, $c\vert_{\varepsilon = 0} \neq 3c_u$ and $\gamma_2 = 0$ \label{scenario3Intro}
	\item $c = -c_v + \varepsilon c_0$ with $c_0 \neq 0$, $c\vert_{\varepsilon = 0} \neq 3c_u$ and $\gamma_2 = \varepsilon\gamma_2^0$ \label{scenario4Intro}
	\item $c = -c_v + \varepsilon c_0$ with $c_0 \neq 0$, $c\vert_{\varepsilon = 0} = 3c_u$ and $\gamma_2 = 0$ \label{scenario5Intro}
\end{enumerate}

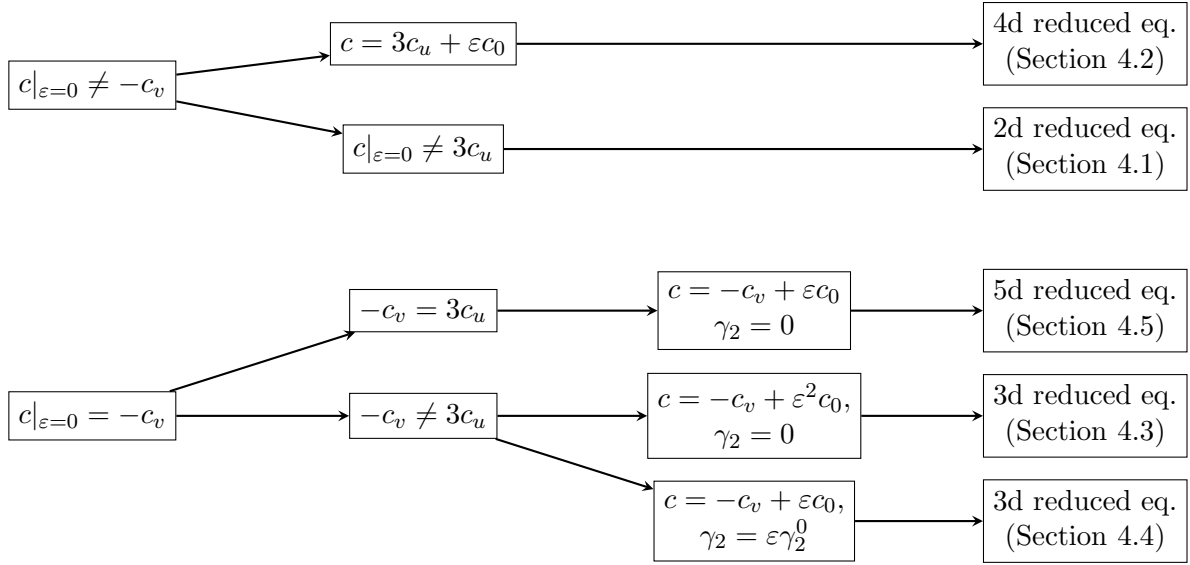
\begin{figure}
	\centering
	\begin{tikzpicture}[node distance = 1.4cm]
		\node (c2) [field] {$c \vert_{\varepsilon = 0} = - c_v$};
		\node (c2s2) [field, right of = c2, xshift = 3cm] {$-c_v \neq 3c_u$};
		\node (c2s2s1) [field, right of = c2s2, xshift = 3cm] {$c = -c_v + \varepsilon^2 c_0$, \\ $\gamma_2 = 0$};
		\node (c2s2s1Res) [field, right of = c2s2s1, xshift = 3cm] {3d reduced eq. \\ (Section \ref{sec:S3})};
		
		\node (c2s1) [field, above of = c2s2] {$-c_v = 3c_u$};
		\node (c2s1Param) [field, above of = c2s2s1] {$c = -c_v + \varepsilon c_0$ \\ $\gamma_2 = 0$};
		\node (c2s1Res) [field, above of = c2s2s1Res] {5d reduced eq. \\ (Section \ref{sec:S5})};
		
		\node (c2s2s2) [field, below of = c2s2s1] {$c = -c_v + \varepsilon c_0$, \\ $\gamma_2 = \varepsilon \gamma_2^0$};
		\node (c2s2s2Res) [field, below of = c2s2s1Res] {3d reduced eq. \\ (Section \ref{sec:S4})};
		
		\node (c1) [field, above of = c2, yshift = 3cm] {$c \vert_{\varepsilon = 0} \neq -c_v$};
		\node (c1s1) [field, above of = c2s1, yshift = 0.75cm] {$c \vert_{\varepsilon = 0} \neq 3c_u$};
		\node (c1s1Res) [field, above of = c2s1Res, yshift = 0.75cm] {2d reduced eq. \\ (Section \ref{sec:S1})};
		
		\node (c1s2) [field, above of = c1s1] {$c = 3c_u + \varepsilon c_0$};
		\node (c1s2Res) [field, above of = c1s1Res] {4d reduced eq. \\ (Section \ref{sec:S2})};

		\draw [arrow] (c2) -- (c2s2);
		\draw [arrow] (c2s2) -- (c2s2s1);
		\draw [arrow] (c2s2s1) -- (c2s2s1Res);
		
		\draw [arrow] (c2s2) -- (c2s2s2);
		\draw [arrow] (c2s2s2) -- (c2s2s2Res);
		
		\draw [arrow] (c2) -- (c2s1);
		\draw [arrow] (c2s1) -- (c2s1Param);
		\draw [arrow] (c2s1Param) -- (c2s1Res);
		
		\draw [arrow] (c1) -- (c1s1);
		\draw [arrow] (c1s1) -- (c1s1Res);
		\draw [arrow] (c1) -- (c1s2);
		\draw [arrow] (c1s2) -- (c1s2Res);
	\end{tikzpicture}
	\caption{An overview of the different scenarios. Starting from the cases that either $c\vert_{\varepsilon = 0}$ equals $-c_v$ or not, we depict how different parameter choices affect the dimension of the reduced equation on the center manifold. Here, we count dimensions after splitting complex variables into their real and imaginary parts. In particular, this means that in all cases with even dimension the conserved center mode is governed by the non-conserved ones on the center manifold and does not contribute an additional equation.}
	\label{fig:Scenarios}
\end{figure}

The scenarios can be split into three subgroups.
The first group contains scenarios \ref{scenario1Intro} and \ref{scenario2Intro}.
In both scenarios the spreading speed $c$ of the modulating traveling front is away from $-c_v$, the linear group velocity corresponding to the conservation law.
It turns out that the center mode of the conservation law is then determined by the center modes of the dispersive Swift-Hohenberg equation, since the only central eigenvalue of the conservation law is a zero eigenvalue, which can be removed by integration.
Therefore, only the central modes of the dispersive Swift-Hohenberg equation are relevant to understand the dynamical behavior on the center manifold.

Scenario \ref{scenario1Intro} (see Section \ref{sec:S1}) then additionally assumes that the spreading speed $c$ is away from $3c_u$, the linear group velocity corresponding to the dispersive Swift-Hohenberg equation.
In this setting, the reduced equation on the center manifold is a one-dimensional, complex-valued equation, which has a gauge invariance reflecting the translational invariance of \eqref{eq:SHDis}--\eqref{eq:ConDis}.
By exploiting the gauge invariance and using polar coordinates it is then sufficient to study a one-dimensional, \emph{real}-valued equation and the existence of heteroclinic orbits is established rigorously.

In contrast, Scenario \ref{scenario2Intro} (see Section \ref{sec:S2}) assumes that the spreading speed $c$ is $\varepsilon$-close to $3c_u$, i.e.~that $c = 3c_u + \varepsilon c_0$.
For this choice, the spatial system reduces to a two-dimensional, complex-valued system on the center manifold.
Neglecting higher order terms, we analyze the system using numerical methods and obtain numerical evidence for the existence of heteroclinic orbits if $c_0$ is larger than some $c_0^\ast > 0$.
Assuming the existence of these heteroclinic connections, we show rigorously that they persist if the higher order terms are included, which establishes the existence of modulating fronts.
Furthermore, we study the behavior if $c_0$ decreases below $c_0^\ast$ using AUTO \cite{auto07p} and in particular find that the system undergoes a Hopf bifurcation at $c_0^\ast$.
Therefore, the amplitude of the modulating front does not decay to zero, but converges to a periodic state, see Figure \ref{fig:hetConnectionsGroupVelocity}.

The second subgroup within the different scenarios contains Scenarios \ref{scenario3Intro} and \ref{scenario4Intro}.
Here, the spreading speed of the modulating traveling front is assumed to be close to the linear group velocity of the conservation law $-c_v$ with respect to $\varepsilon$.
In contrast to the first subgroup encompassing Scenarios \ref{scenario1Intro} and \ref{scenario2Intro}, the central mode of the conservation law is no longer governed by the central modes of the dispersive Swift-Hohenberg equation.
This is due to an additional central eigenvalue of the conservation law, which cannot be removed by integration.

Scenario \ref{scenario3Intro} deals with the case that $c$ converges quadratically to $-c_v$ as $\varepsilon$ tends to zero, i.e.~$c = -c_v + \varepsilon^2 c_0$ for some $c_0 \neq 0$.
In addition, the coupling parameter $\gamma_2$ is set to zero since $\partial_x(u^2)$ in \eqref{eq:ConDis} leads to coefficients in the reduced system which blow-up at least like $\varepsilon^{-2}$ for $\varepsilon \rightarrow 0$ in the natural scaling.
Then, the spatial system reduces to a system of a one-dimensional, complex-valued and a one-dimensional real-valued equation.
If the remaining coupling term $\gamma_1$, see \eqref{eq:ConDis}, is removed, the system degenerates to the reduced equation treated in Scenario \ref{scenario1Intro}, which exhibits heteroclinic connections.
We then show that these connections persist if the coupling parameter $\gamma_1$ is close to zero.

In Scenario \ref{scenario4Intro}, we then assume that $c$ converges linearly rather than quadratically to $-c_v$ as $\varepsilon$ tends to zero, that is $c = -c_v + \varepsilon c_0$ for some $c_0 \neq 0$.
This setting also allows to set $\gamma_2 = \curlO(\varepsilon)$.
In this case the central eigenvalues are of different order in $\varepsilon$ and therefore, it is not possible to choose a re-scaling such that the linear part of the reduced system on the center manifold is independent of $\varepsilon$, see Remarks \ref{rem:spatialRescaling} and \ref{rem:spatialRescalingS3}.
The reduced system then takes the form of a fast-slow system, that is, it is of the form
\begin{align*}
	\dot{f} &= F(f,g,\varepsilon), \\
	\varepsilon\dot{g} &= G(f,g,\varepsilon),
\end{align*}
where $\dot{(\cdot)}$ denotes a ``time'' derivative and $F,G$ are suitable smooth functions.
We then use geometric singular perturbation theory (see e.g.~\cite{kuehn15}) to construct persistent heteroclinic connections on the center manifold.

Finally, the remaining Scenario \ref{scenario5Intro} deals with the case where the spreading speed $c$ is both close to $3c_u$ and $-c_v$ and in particular assumes that $3c_u = -c_v$.
This leads to a reduced system on the center manifold, which couples a two-dimensional, complex-valued system for the central modes of the dispersive Swift-Hohenberg equation to a one-dimensional, real-valued equation for the central mode of the conservation law.
Here, again we have to set $\gamma_2 = 0$ to avoid the presence of coefficients, which exhibit a blow-up as $\varepsilon \rightarrow 0$.
Similar to Scenario \ref{scenario3Intro}, we decouple the system by setting $\gamma_1 = 0$.
Then, the system on the center manifold reduces to the two-dimensional system discussed in Scenario \ref{scenario2Intro}.
Therefore, we have numerical evidence for the existence of the desired heteroclinic connections and can show their persistence if the coupling parameter $\gamma_1$ is sufficiently close to zero.

To conclude, in Scenarios \ref{scenario1Intro}, \ref{scenario3Intro} and \ref{scenario4Intro}, we rigorously show the existence of modulating traveling fronts in \eqref{eq:SHDis}--\eqref{eq:ConDis}, which we summarize formally in the following result (see Theorems \ref{thm:S1}, \ref{thm:S3} and \ref{thm:S4} for the precise formulation).

\begin{theorem}\label{thm:main1}
	Let the dispersion coefficient $c_u$ be non-zero and assume that $c\vert_{\varepsilon = 0} \neq c_u$.
	Then, for parameters as in Scenario \ref{scenario1Intro}, \ref{scenario3Intro} or \ref{scenario4Intro}, $\varepsilon > 0$ sufficiently small and $\gamma_1, \gamma_2$ close to zero, the system \eqref{eq:SHDis}--\eqref{eq:ConDis} has a family of modulating traveling front solutions.
\end{theorem}

Additionally, in the Scenarios \ref{scenario2Intro} and \ref{scenario5Intro}, we have numerical evidence for the existence of heteroclinic connections in the reduced equation after neglecting higher order terms in $\varepsilon$.
Furthermore, assuming the existence of these heteroclinics, we can rigorously show their persistence with respect to small perturbations in $\varepsilon$ and thus, via the center manifold theory, obtain the following result, which we again state in a formal way and refer to Theorems \ref{thm:S2} and \ref{thm:S5} for the full details.

\begin{theorem}\label{thm:main2}
	Let the dispersion coefficient $c_u$ be non-zero and assume $c\vert_{\varepsilon = 0} \neq c_u$.
	Furthermore, assume that Scenario \ref{scenario2Intro} or \ref{scenario5Intro} holds and that the reduced system up to higher order terms has a suitable heteroclinic orbit.
	Then, for $\varepsilon > 0$ sufficiently small and $\gamma_1, \gamma_2 \in \R$ close to zero, the system \eqref{eq:SHDis}--\eqref{eq:ConDis} has a family of modulating traveling front solutions.
\end{theorem}

\begin{remark}\label{rem:nonMonotonicDecay}
	It turns out that the amplitude of the modulating traveling front is in general not monotonically decreasing to zero.
	This is generally the case in the reflection symmetric case \cite{hilder20}.
	However, due to the presence of complex-valued coefficients in the reduced equation in the dispersive case, the eigenvalues of the linearization about zero are generically complex-valued.
	Therefore, the heteroclinic orbit, which describes the amplitude of the modulating front, exhibits an oscillatory decay to zero.
	This is depicted in Figure \ref{fig:modfront}, see also Figure \ref{fig:hetConnectionsGroupVelocity}.
\end{remark}

\begin{remark}\label{rem:speedCloseToPhaseVel}
	We highlight that in both Theorems \ref{thm:main1} and \ref{thm:main2} we assume that $c\vert_{\varepsilon = 0} \neq c_u$.
	This guarantees that there are only finitely many eigenvalues on the imaginary axis at $\varepsilon = 0$ and that there is an $\varepsilon$-independent spectral gap about the imaginary axis.
	This is pivital for the application of the center manifold theorem.
	In fact, if we set $c\vert_{\varepsilon = 0} = c_u$ there are infinitely many eigenvalues on the imaginary axis for $\varepsilon = 0$, see the proofs of Lemmas \ref{lem:specLSH} and \ref{lem:specLcon}.
	
	This also shows that it is crucial for the analysis that the reflection symmetry in the Swift-Hohenberg equation \eqref{eq:SHDis} is broken via a third order derivative rather than a first order one.
	If the symmetry is broken via a $c_u \partial_x$-term, the corresponding linear phase and group velocity coincide.
	Therefore, in this case the results in Theorem \ref{thm:main2} cannot be established with the methods used here.
	
	It is worth pointing out that the setting $c\vert_{\varepsilon = 0} = c_u$ is comparable to \cite{hilder20}, which constructs modulating fronts for the reflection symmetric system \eqref{eq:SH}--\eqref{eq:con}, which have a spreading speed of order $\varepsilon$ and connect to spatially periodic waves with vanishing phase velocity (due to the reflection symmetry).
	The construction for \eqref{eq:SH}--\eqref{eq:con} then uses the observation that, although for $\varepsilon = 0$ there are infinitely many purely imaginary eigenvalues, these eigenvalues leave the imaginary axis with different speeds for $\varepsilon > 0$.
	It then turns out that there is a finite number of ``central'' eigenvalues with real part of order $\varepsilon$, while the remaining eigenvalues have a real part of order $\sqrt{\varepsilon}$, see \cite[Figure 4]{hilder20}.
	However, if a similar procedure is applicable in the dispersive system \eqref{eq:SHDis}--\eqref{eq:ConDis} for $c\vert_{\varepsilon = 0} = c_u$ is still an open question.
\end{remark}

\subsection{Outline}

The paper is organized as follows.
In Section \ref{sec:travelingWaves}, we establish the existence of bifurcating traveling wave solutions in the system \eqref{eq:SHDis}--\eqref{eq:ConDis} and show that their phase velocity is $\varepsilon^2$-close to $c_u$, the linear phase velocity of the dispersive Swift-Hohenberg equation \eqref{eq:SHDis}.
Following, Section \ref{sec:spatialDynamics} contains the derivation of the spatial dynamics formulation, a spectral analysis (see Lemmas \ref{lem:specLSH} and \ref{lem:specLcon}) and the center manifold result.
Afterwards, the derivation of the reduced equation in Scenarios \ref{scenario1Intro}--\ref{scenario5Intro} and the existence of persistent heteroclinic connections is done in Section \ref{sec:redEqAndHetConnections}.

\section{Spatially periodic traveling wave solutions}\label{sec:travelingWaves}

We first prove the existence of spatially periodic traveling wave solutions in the system \eqref{eq:SHDis}--\eqref{eq:ConDis}.
The proof is based on center manifold theory.
Since the system lacks reflection symmetry due to the dispersive terms, the traveling waves have a non-zero phase velocity.
As discussed in the introduction, the phase velocity is expected to be close to the linear phase velocity of the dispersive Swift-Hohenberg equation $c_\text{lin,phase,u}(k_c) = c_u$, where $k_c = 1$.
Therefore, we make the ansatz $\cphase = c_u + \varepsilon^2 \omega_0^\ast$ and use the correction $\omega_0^\ast$ to guarantee the existence of non-trivial, real fixed points on the center manifold.

\begin{lemma}\label{lem:travelingWaves}
	Let $B \in \R$ and $\alpha_0 \in \R$ such that $B + \alpha_0 > 0$.
	Then, there exist $\varepsilon_0 > 0$, $\gamma^\ast > 0$ such that for all $\varepsilon \in (0,\varepsilon_0)$ and $\gamma_1,\gamma_2 \in (-\gamma^\ast, \gamma^\ast)$ there exist $A^\ast \in \R$ and $\omega_0^\ast \in \R$ given by
	\begin{align*}
		A^\ast &= \pm \sqrt{\dfrac{(9+4c_u^2)(B+\alpha_0)}{4(7+3c_u^2)}} + \curlO(\varepsilon^2 (1 + \snorm{\gamma_1} + \snorm{\gamma_2})), \\
		\omega_0^\ast &= - \dfrac{c_u(B + \alpha_0)}{6(7+3c_u^2)} + \curlO(\varepsilon^2 (1 + \snorm{\gamma_1} + \snorm{\gamma_2})),
	\end{align*}
	such that \eqref{eq:SHDis}--\eqref{eq:ConDis} has a spatially periodic traveling wave solution of the form
	\begin{align*}
		\utw(p) &= \varepsilon A^\ast e^{ip} + c.c. + \curlO(\varepsilon^3 (1 + \snorm{\gamma_1} + \snorm{\gamma_2})), \\
		\vtw(p) &= \varepsilon^2 B + \varepsilon^2 \dfrac{-4\gamma_1 + 2i\gamma_2}{4 - 2i(c_u+c_v)}(A^\ast)^2 e^{2ip} + c.c. + \curlO(\varepsilon^3 (1 + \snorm{\gamma_1} + \snorm{\gamma_2})),
	\end{align*}
	with $p = x - \cphase t$ and phase velocity $\cphase = c_u + \varepsilon^2 \omega_0^\ast$.
	Here, $c.c.$ denotes the complex conjugated terms.
\end{lemma}
\begin{proof}
	To apply center manifold theory, we first rewrite the system \eqref{eq:SHDis}--\eqref{eq:ConDis} in a co-moving frame with speed $\cphase$, which yields
	\begin{align*}
		\partial_t u &= -(1+\partial_p^2)^2 u + \varepsilon^2 \alpha_0 u + c_u \partial_p^3 u + \cphase\partial_p u + uv + \dfrac{1}{2} \partial_p (u^2) - u^3, \\
		\partial_t v &= \partial_p^2 v + (c_v + \cphase) \partial_p v + \gamma_1\partial_p^2(u^2) + \gamma_2 \partial_p(u^2),
	\end{align*}
	with $p = x-\cphase t$.
	We now apply the center manifold theory in \cite{haragusIooss11} using the spaces
	\begin{align*}
		\curlZ_B &= \left\{(u,v) \in H^4_\text{per} \times H^2_\text{per} \,:\, u_n = \bar{u}_{-n}, v_n = \bar{v}_{-n}, v_0 = \varepsilon^2 B\right\}, \\
		\curlY_B &= \left\{(u,v) \in H^2_\text{per} \times H^2_\text{per} \,:\, u_n = \bar{u}_{-n}, v_n = \bar{v}_{-n}, v_0 = \varepsilon^2 B\right\}, \\
		\curlX_B &= \left\{(u,v) \in H^0_\text{per} \times H^0_\text{per} \,:\, u_n = \bar{u}_{-n}, v_n = \bar{v}_{-n}, v_0 = \varepsilon^2 B\right\},
	\end{align*}
	with $B \in \R$ and $H_\text{per}^n$ the space of $2\pi$-periodic $H^n$-functions for $n \in \N$.
	Here $H^n$ denotes the $L^2$-based Sobolev spaces, see \cite{schneiderUecker17}.
	We highlight that these spaces are invariant under the dynamics of \eqref{eq:SHDis}--\eqref{eq:ConDis}.
	The linear part $L$ is given by
	\begin{align*}
		L = \begin{pmatrix}
			-(1+\partial_p^2)^2 + \varepsilon^2 \alpha_0 + c_u \partial_p^3 + \cphase\partial_p & 0 \\
			0 & \partial_p^2 + (c_v + \cphase)\partial_p
		\end{pmatrix}
	\end{align*}
	and its spectrum can be calculated explicitly using Fourier transform as
	\begin{align*}
		\lambda_{u,n} &= -(1-n^2)^2 + \varepsilon^2 \alpha_0 + i(n-n^3)c_u + in\varepsilon^2\omega_0, \quad n \in \Z, \\
		\lambda_{v,n} &= -n^2 + in(c_u + c_v +\varepsilon^2\omega_0), \quad n \in \Z \setminus\{0\},
	\end{align*}
	where we made the ansatz $\cphase = c_u + \varepsilon^2 \omega_0$.
	Note that since we consider spaces with fixed mode $v_0 = \varepsilon^2 B$, we find that $\lambda_{v,0} = 0$ is not part of the spectrum.
	Thus, the spectrum can be decomposed in a central part $\sigma_c(L) = \{\lambda_{u,\pm1}\}$ and a remaining hyperbolic part $\sigma_h(L)$.
	Since we also have a resolvent estimate by explicit calculation in Fourier space and $\curlZ_B$, $\curlY_B$ and $\curlX_B$ are Hilbert spaces the center manifold results available in \cite{haragusIooss11} are applicable.
	Therefore, we introduce the following coordinates on the center manifold
	\begin{align*}
		u(t) &= \varepsilon A(\varepsilon^2 t) e^{ip} + c.c. + h_u(\varepsilon A), \\
		v(t) &= \varepsilon^2 B + h_v(\varepsilon A).
	\end{align*}
	To obtain an approximation of $h_u$ and $h_v$ we apply \cite[Corollary 2.12]{haragusIooss11} and use that $h_u, h_v = \curlO(\varepsilon^2)$, since they depend at least quadratically on $\varepsilon A$.
	Thus, we obtain that $h_u$ satisfies
	\begin{align*}
		0 &= -(1+\partial_p^2)^2 h_u + c_u (\partial_p + \partial_p^3) h_u + \dfrac{1}{2}\partial_p(\varepsilon^2 A^2 e^{2ip} + c.c. + 2\varepsilon^2 \snorm{A}^2) + \curlO(\varepsilon^4) \\
		&= -(1+\partial_p^2)^2 h_u + c_u (\partial_p + \partial_p^3) h_u + \varepsilon^2 (i A^2 e^{2ip} + c.c.) + \curlO(\varepsilon^4).
	\end{align*}
	Furthermore, $h_v$ satisfies
	\begin{align*}
		0 &= \partial_p^2 h_v + (c_v + c_u) \partial_p h_v + (\gamma_1 \partial_p^2 + \gamma_2 \partial_p)\left(\varepsilon^2 A^2 e^{2ip} + c.c. + \varepsilon^2 \snorm{A}^2\right) + \curlO(\varepsilon^4) \\
		&= \partial_p^2 h_v + (c_v + c_u) \partial_p h_v + (-4\gamma_1 + 2i\gamma_2) (A^2 e^{2ip} + c.c.) + \curlO(\varepsilon^4).
	\end{align*}
	Hence, we find
	\begin{align*}
		h_u &= \varepsilon^2 \dfrac{i}{9+6ic_u} A^2 e^{2ip} + c.c. + \curlO(\varepsilon^3), \\
		h_v &= \varepsilon^2 \dfrac{-2\gamma_1 + i\gamma_2}{2-i(c_u+c_v)} A^2 e^{2ip} + c.c. + \curlO(\varepsilon^3).
	\end{align*}
	Therefore, we arrive at a reduced equation on the center manifold, which is given by
	\begin{align*}
		\partial_T A = (\alpha_0 + i\omega_0) A + AB + \left(-3-\dfrac{1}{9+6ic_u} + \dfrac{-2\gamma_1 + i\gamma_2}{2-i(c_u+c_v)}\right)A\snorm{A}^2 + g(A,\varepsilon),
	\end{align*}
	with $g(A,\varepsilon) = \curlO(\varepsilon^2)$.
	For $(\varepsilon,\gamma_1,\gamma_2) = 0$, this equation has a family of non-trivial stationary solutions if and only if
	\begin{align*}
		\snorm{A}^2 = \left(3+\dfrac{1}{9+6ic_u}\right)^{-1}(\alpha_0+B + i\omega_0)
	\end{align*}
	is real and positive.
	Therefore, we calculate
	\begin{align*}
		\Im\left[\left(3+\dfrac{1}{9+6ic_u}\right)^{-1}(\alpha_0+B + i\omega_0)\right] &= \dfrac{3c_u(\alpha_0+B)+3(42+18c_u^2)\omega_0}{392+161c_u^2}.
	\end{align*}
	This expression vanishes if and only if
	\begin{align*}
		\omega_0 = \omega_0^\ast := -\dfrac{c_u(B+\alpha_0)	}{6(7+3c_u^2)}.
	\end{align*}
	Then, the expression for $\snorm{A}^2$ simplifies to
	\begin{align*}
		\snorm{A}^2 = \snorm{A^\ast}^2 := \dfrac{(9+4c_u^2)(B+\alpha_0)}{4(7+3c_u^2)},
	\end{align*}
	which is positive since $B+\alpha_0 > 0$ by assumption.
	Therefore, we have a non-trivial, stationary solution of the reduced equation for $(\varepsilon, \gamma_1, \gamma_2) = (0,0,0)$.
	
	It remains to show that this solution persists for $(\varepsilon,\gamma_1,\gamma_2)$ close to zero.
	Recalling that the system \eqref{eq:SHDis}--\eqref{eq:ConDis} is translationally invariant, we find that the reduced equation is invariant with respect to $A \mapsto A e^{i\phi}$.
	Therefore, we may assume that the solution is real and positive and satisfies
	\begin{align*}
		0 &= (\alpha_0 + B) A + a^\text{cub}(\gamma_1,\gamma_2) A^3 + \Re(g(A,\varepsilon)) =: G_1(A,\omega_0,\varepsilon,\gamma_1,\gamma_2), \\
		0 &= \omega_0 A + b^\text{cub}(\gamma_1,\gamma_2) A^3 + \Im(g(A,\varepsilon)) =: G_2(A,\omega_0,\varepsilon,\gamma_1,\gamma_2),
	\end{align*}
	where the coefficients are given by
	\begin{align*}
		a^\text{cub}(\gamma_1,\gamma_2) &= \Re\left(-3-\dfrac{1}{9+6ic_u} + \dfrac{-2\gamma_1 + i\gamma_2}{2-i(c_u+c_v)}\right), \\
		b^\text{cub}(\gamma_1,\gamma_2) &= \Im\left(-3-\dfrac{1}{9+6ic_u} + \dfrac{-2\gamma_1 + i\gamma_2}{2-i(c_u+c_v)}\right).
	\end{align*}
	Additionally, we define $G := (G_1,G_2)^T$.
	Then, by construction it holds that $G(A^\ast,\omega_0^\ast,0,0,0) = 0$ and we have
	\begin{align*}
		D_{(A,\omega_0)} G\vert_{(A,\omega_0,\varepsilon,\gamma_1,\gamma_2)=(A^\ast,\omega_0^\ast,0,0,0)} = \begin{pmatrix}
			\alpha_0 + B + 3 a^\text{cub}(0,0) (A^\ast)^2 & 0 \\
			\omega_0^\ast + 3 b^\text{cub}(0,0) (A^\ast)^2 & A^\ast
		\end{pmatrix}.
	\end{align*}
	Since $A^\ast > 0$, this matrix is invertible if and only if
	\begin{align*}
		\alpha_0 + B + 3 a^\text{cub}(0,0) (A^\ast)^2 = -2(B+\alpha_0) \neq 0,
	\end{align*}
	which is true since $B+ \alpha_0 >0$ by assumption.
	Therefore, the solutions persist by the implicit function theorem and we obtain the statement of the lemma.
\end{proof}

\begin{remark}
	Since the system \eqref{eq:SHDis}--\eqref{eq:ConDis} is translationally invariant, Lemma \ref{lem:travelingWaves} actually establishes a family of periodic solutions $(u_{\text{tw},x_0},v_{\text{tw},x_0})(p) = (\utw,\vtw)(p+x_0)$ for all $x_0 \in [0,2\pi)$.
	Although this paper is restricted to the construction of modulating fronts, which connect to $(\utw,\vtw)$, fronts connecting to any other periodic solution $(u_{\text{tw},x_0},v_{\text{tw},x_0})$ can be obtained in the same fashion.
\end{remark}

\section{Spatial dynamics and center manifold result}\label{sec:spatialDynamics}

We now proceed in the construction of modulating traveling front solutions for the system \eqref{eq:SHDis}--\eqref{eq:ConDis}, which connect the periodic state obtained in Lemma \ref{lem:travelingWaves} to the homogeneous ground state $(u,v) = (0,0)$.
The aim of this section is to obtain the spatial dynamics formulation and establish a center manifold result.
Therefore, we make the following ansatz
\begin{align}
	(u,v)(t,x) = (\curlU,\curlV)(x-ct,x-\cphase t) = (\curlU,\curlV)(\xi,p),
	\label{eq:modfrontAnsatz}
\end{align}
which is assumed to be $2\pi$-periodic with respect to its second argument and satisfies the asymptotic conditions
\begin{align*}
	\lim_{\xi \rightarrow -\infty} (\curlU,\curlV)(\xi,p) = (\utw,\vtw)(p) \text{ and }\lim_{\xi \rightarrow \infty} (\curlU,\curlV)(\xi,p) = (0,0).
\end{align*}
Here, $\cphase = c_u + \varepsilon^2 \omega_0^\ast$ is the phase velocity of the periodic solution constructed in Lemma \ref{lem:travelingWaves} and $c \in \R$ is the velocity of the modulating front.
Inserting this ansatz into \eqref{eq:SHDis}--\eqref{eq:ConDis}, we obtain
\begin{subequations}
	\begin{align}
		-c\partial_\xi \curlU - \cphase \partial_p \curlU &= -(1+(\partial_\xi + \partial_p)^2)^2 \curlU + \varepsilon^2 \alpha_0 \curlU + c_u (\partial_\xi + \partial_p)^3 \curlU + \curlU\curlV + \dfrac{1}{2}(\partial_\xi + \partial_p)(\curlU^2) - \curlU^3, \label{eq:modfrontSysSH}\\
		-c\partial_\xi \curlV - \cphase \partial_p \curlV &= (\partial_\xi + \partial_p)^2 \curlV + c_v (\partial_\xi + \partial_p) \curlV + \gamma_1 (\partial_\xi + \partial_p)^2(\curlU^2) + \gamma_2(\partial_\xi + \partial_p)(\curlU^2). \label{eq:modfrontSysCon}
	\end{align}
\end{subequations}
Using the periodicity with respect to $p$ we can expand $(\curlU,\curlV)$ into a Fourier series in $p$, that is
\begin{align*}
	(\curlU,\curlV)(\xi,p) = \sum_{n \in \Z} (\curlU_n,\curlV_n)(\xi) e^{inp}.
\end{align*}
We point out that since the solution $(\curlU,\curlV)$ is real-valued, the Fourier coefficients satisfy $\curlU_n = \overline{\curlU_{-n}}$ and $\curlV_n = \overline{\curlV_{-n}}$ for all $n \in \Z$.
Here, $\overline{(\cdot)}$ denotes complex conjugation.
Inserting this Fourier series in \eqref{eq:modfrontSysSH}--\eqref{eq:modfrontSysCon} and writing the resulting equation as a first order system in $\xi$, we obtain the spatial dynamics formulation
\begin{align}
	\partial_\xi \begin{pmatrix}
		U_n \\ V_n
	\end{pmatrix} = \begin{pmatrix}
		L_n^\text{SH} & 0 \\
		0 & L_n^\text{con}
	\end{pmatrix} \begin{pmatrix}
		U_n \\ V_n
	\end{pmatrix} + \curlN_n(U,V)
	\label{eq:spatDyn}
\end{align}
with $n \in \Z$.
Here, $(U_n,V_n) = (U_{n0},U_{n1},U_{n2},U_{n3},V_{n0},V_{n1})$ with $U_{nj} = \partial_\xi^j \curlU_n$ for $j = 0,1,2,3$ and $V_{nj} = \partial_\xi^j \curlV_n$ for $j = 0,1$.
The matrices $L_n^\text{SH} \in \C^{4\times 4}$ and $L_n^\text{con} \in \C^{2\times 2}$ are given by
\begin{align*}
	L_n^\text{SH} = \begin{pmatrix}
		0 & 1 & 0 & 0 \\
		0 & 0 & 1 & 0 \\
		0 & 0 & 0 & 1 \\
		A_n & B_n & C_n & D_n
	\end{pmatrix}, \qquad L_n^\text{con} = \begin{pmatrix}
		0 & 1 \\
		E_n & F_n
	\end{pmatrix}
\end{align*}
with
\begin{align*}
	A_n &= \varepsilon^2 \alpha_0 - ic_u n^3 + in\cphase - (1-n^2)^2, \\
	B_n &= -3c_u n^2 + c + 4i(n^3-n), \\
	C_n &= 3inc_u + 6n^2 - 2, \\
	D_n &= c_u - 4in, \\
	E_n &= -inc_v + n^2 - in\cphase, \\
	F_n &= -(c+c_v+2in).
\end{align*}
Additionally, the nonlinearity $\curlN_n(U,V)$ is of the form $(0,0,0,\curlN_n^\text{SH}(U,V),0,\curlN_n^\text{con}(U,V))^T$ with $\curlN_n^\text{SH}$ and $\curlN_n^\text{con}$ given by
\begin{subequations}
\begin{align}
	\curlN_n^\text{SH}(U,V) &= \sum_{k+j=n} U_{k0} V_{j0} + \sum_{k+j=n} U_{k0}(U_{j1} + ijU_{j0}) - \sum_{k+j+l=n} U_{k0}U_{j0}U_{l0}, \label{eq:nonlinSH}\\
	\curlN_n^\text{con}(U,V) &= -2\gamma_1 \sum_{k+j = n} (U_{k0} U_{j2} + U_{k1} U_{j1}) - 2\gamma_1 in \sum_{k+j = n} U_{k0} U_{j1} + 2\gamma_1 n^2 \sum_{k+j = n} U_{k0} U_{j0} \nonumber\\
	&\qquad\qquad- 2\gamma_2 \sum_{k+j = n} U_{k0} U_{j1} - \gamma_2 in \sum_{k+j=n} U_{k0} U_{j0}. \label{eq:nonlinCon}
\end{align}
\end{subequations}

\subsection{Function spaces and spectrum}

Following the setup in \cite{eckmannWayne91,hilder20} we define the function space
\begin{align*}
	H^\ell(\C^m) := \left\{U(p) = \sum_{n \in \Z} U_n e^{inp}, U_n \in \C^m \,:\, \norm[H^\ell]{U} < \infty\right\},
\end{align*}
where $\ell > 0$, $m \in \N$ and the norm $\norm[H^\ell]{\cdot}$ is defined by
\begin{align*}
	\norm[H^\ell]{U}^2 := \sum_{n \in \Z} (1+n^2)^\ell \snorm{U_n}^2.
\end{align*}
This is a Hilbert space for $\ell > 0$ and a Banach algebra for $\ell > 1/2$.
Next, we define the space $\curlE_{\ell_1,\ell_2} := H^{\ell_1}(\C^4) \times H^{\ell_2}(\C^2)$, which is equipped with the norm $\norm[\curlE_{\ell_1,\ell_2}]{(U,V)} = \norm[H^{\ell_1}]{U} + \norm[H^{\ell_2}]{V}$ for $U \in H^{\ell_1}(\C^4)$ and $V \in H^{\ell_2}(\C^2)$.
Furthermore, in the special case that $\ell_1 = \ell_2$ we write $\curlE_\ell = \curlE_{\ell,\ell}$.
In this setting, the spatial system \eqref{eq:spatDyn} then reads as
\begin{align}
	\partial_\xi \begin{pmatrix}
		U \\ V
	\end{pmatrix} = L \begin{pmatrix}
		U \\ V
	\end{pmatrix} + \curlN(U,V)
	\label{eq:spatDynFinal}
\end{align}
with $(U,V) \in \curlE_\ell$, the linear operator $L : \curlE_{\ell+4,\ell+2} \rightarrow \curlE_\ell$ and the nonlinearity $\curlN : \curlE_{\ell+4,\ell+2} \rightarrow \curlE_{\ell+2}$, which are given by
\begin{align*}
	L\begin{pmatrix}
		U \\ V
	\end{pmatrix} = \sum_{n \in \Z} \begin{pmatrix}
		L_n^\text{SH} & 0 \\
		0 & L_n^\text{con}
	\end{pmatrix} \begin{pmatrix}
		U_n \\ V_n
	\end{pmatrix} e^{inp} \text{ and } \curlN(U,V) = \sum_{n \in \Z} \curlN_n(U,V) e^{inp}.
\end{align*}
Utilizing the diagonal structure of $L$, its spectrum is given by the union of the spectra of $L_n^\text{SH}$ and $L_n^\text{con}$ for all $n \in \Z$.
These are given in the following lemmas and the results are schematically depicted in Figure \ref{fig:spectralSituations}, see also \cite[Figure 6]{hilder20}.

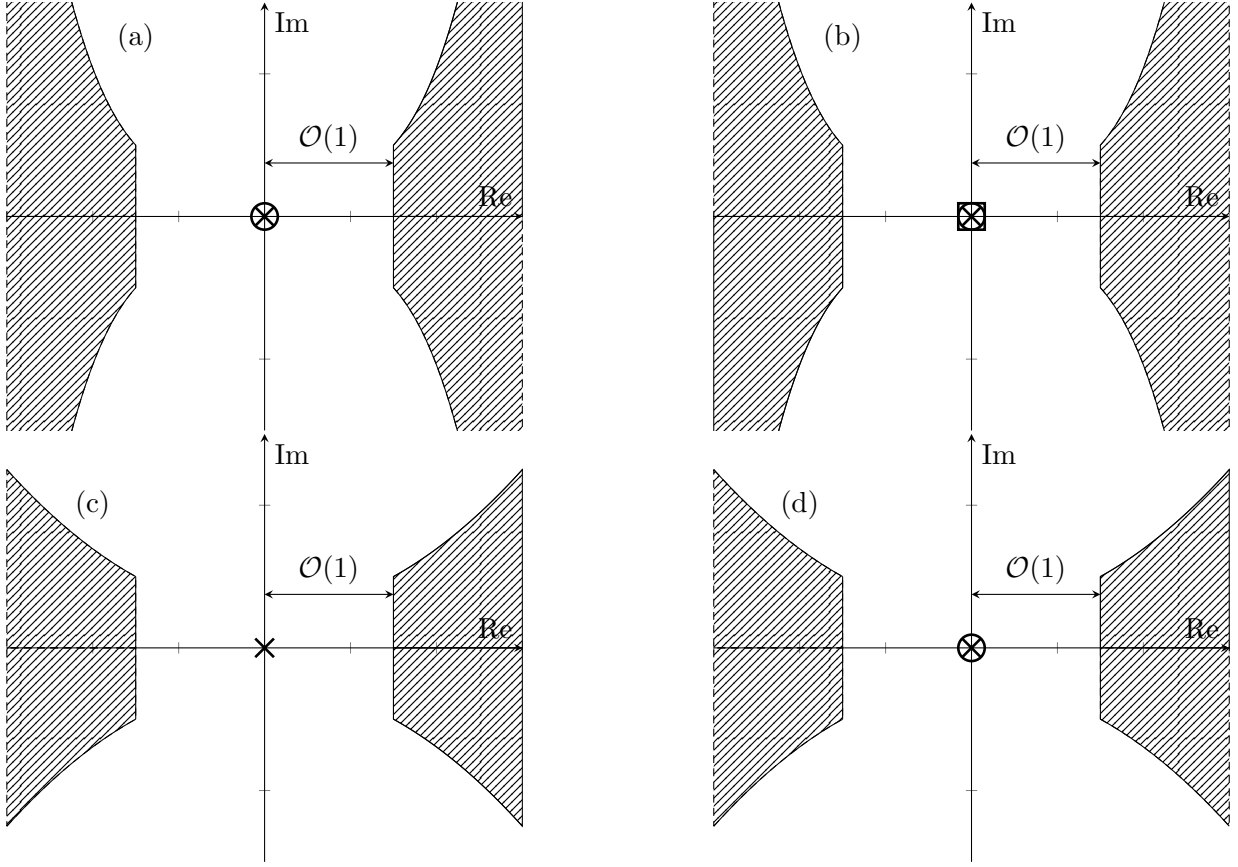
\begin{figure}
	\begin{tikzpicture}
		\begin{axis}[
				xmin = -3,
				xmax = 3,
				ymin = -3,
				ymax = 3,
				xticklabels = \empty,
				yticklabels = \empty,
				xlabel = $\Re$,
				ylabel = $\Im$,
				x label style={anchor=south west},
				axis x line = center,
				axis y line = center,
			]
			\addplot[mark=x,only marks, mark size = 5pt, mark options={line width = 1pt}] coordinates {(0,0)};
			\addplot[mark=o,only marks, mark size = 5pt, mark options={line width = 1pt}] coordinates {(0,0)};
			
			\node at (axis cs:-1.5,2.5){(a)};
			
			\draw (axis cs:1.5,1)--(axis cs:1.5,-1);
			\draw (axis cs:-1.5,1)--(axis cs:-1.5,-1);
			
			\draw[stealth-stealth] (axis cs: 0,0.75) -- (axis cs: 1.5,0.75) node[midway,above]{$\curlO(1)$};
			
			\addplot[mark=none,samples=100,domain=1.5:3] {0.5*(x/1.5)^4+0.5};
			\addplot[mark=none,samples=100,domain=1.5:3] {-0.5*(x/1.5)^4-0.5};
			
			\addplot+[mark=none,domain=1.5:3,samples=100,%
              pattern=north east lines,%
              draw = black,
              pattern color=black]%
              {0.5*(x/1.5)^4+0.5}
              \closedcycle;
            \addplot+[mark=none,domain=1.5:3,samples=100,%
              pattern=north east lines,%
              draw = black,
              pattern color=black]%
              {-0.5*(x/1.5)^4-0.5}
              \closedcycle;
              
            \addplot[mark=none,samples=100,domain=-3:-1.5] {0.5*(x/1.5)^4+0.5};
			\addplot[mark=none,samples=100,domain=-3:-1.5] {-0.5*(x/1.5)^4-0.5};
			
			\addplot+[mark=none,domain=-3:-1.5,samples=100,%
              pattern=north east lines,%
              draw = black,
              pattern color=black]%
              {0.5*(x/1.5)^4+0.5}
              \closedcycle;
            \addplot+[mark=none,domain=-3:-1.5,samples=100,%
              pattern=north east lines,%
              draw = black,
              pattern color=black]%
              {-0.5*(x/1.5)^4-0.5}
              \closedcycle;
		\end{axis}
	\end{tikzpicture}
	\hfill
	\begin{tikzpicture}
		\begin{axis}[
				xmin = -3,
				xmax = 3,
				ymin = -3,
				ymax = 3,
				xticklabels = \empty,
				yticklabels = \empty,
				xlabel = $\Re$,
				ylabel = $\Im$,
				x label style={anchor=south west},
				axis x line = center,
				axis y line = center,
			]
			\addplot[mark=x,only marks, mark size = 5pt, mark options={line width = 1pt}] coordinates {(0,0)};
			\addplot[mark=o,only marks, mark size = 5pt, mark options={line width = 1pt}] coordinates {(0,0)};
			\addplot[mark=square,only marks, mark size = 5pt, mark options={line width = 1pt}] coordinates {(0,0)};
			
			\node at (axis cs:-1.5,2.5){(b)};
			
			\draw (axis cs:1.5,1)--(axis cs:1.5,-1);
			\draw (axis cs:-1.5,1)--(axis cs:-1.5,-1);
			
			\draw[stealth-stealth] (axis cs: 0,0.75) -- (axis cs: 1.5,0.75) node[midway,above]{$\curlO(1)$};
			
			\addplot[mark=none,samples=100,domain=1.5:3] {0.5*(x/1.5)^4+0.5};
			\addplot[mark=none,samples=100,domain=1.5:3] {-0.5*(x/1.5)^4-0.5};
			
			\addplot+[mark=none,domain=1.5:3,samples=100,%
              pattern=north east lines,%
              draw = black,
              pattern color=black]%
              {0.5*(x/1.5)^4+0.5}
              \closedcycle;
            \addplot+[mark=none,domain=1.5:3,samples=100,%
              pattern=north east lines,%
              draw = black,
              pattern color=black]%
              {-0.5*(x/1.5)^4-0.5}
              \closedcycle;
              
            \addplot[mark=none,samples=100,domain=-3:-1.5] {0.5*(x/1.5)^4+0.5};
			\addplot[mark=none,samples=100,domain=-3:-1.5] {-0.5*(x/1.5)^4-0.5};
			
			\addplot+[mark=none,domain=-3:-1.5,samples=100,%
              pattern=north east lines,%
              draw = black,
              pattern color=black]%
              {0.5*(x/1.5)^4+0.5}
              \closedcycle;
            \addplot+[mark=none,domain=-3:-1.5,samples=100,%
              pattern=north east lines,%
              draw = black,
              pattern color=black]%
              {-0.5*(x/1.5)^4-0.5}
              \closedcycle;
		\end{axis}
	\end{tikzpicture}
	
	\begin{tikzpicture}
		\begin{axis}[
				xmin = -3,
				xmax = 3,
				ymin = -3,
				ymax = 3,
				xticklabels = \empty,
				yticklabels = \empty,
				xlabel = $\Re$,
				ylabel = $\Im$,
				x label style={anchor=south west},
				axis x line = center,
				axis y line = center,
			]
			\addplot[mark=x,only marks, mark size = 5pt, mark options={line width = 1pt}] coordinates {(0,0)};
			
			\node at (axis cs:-2,2){(c)};
			
			\draw (axis cs:1.5,1)--(axis cs:1.5,-1);
			\draw (axis cs:-1.5,1)--(axis cs:-1.5,-1);
			
			\draw[stealth-stealth] (axis cs: 0,0.75) -- (axis cs: 1.5,0.75) node[midway,above]{$\curlO(1)$};
			
			\addplot[mark=none,samples=100,domain=1.5:3] {0.5*(x/1.5)^2+0.5};
			\addplot[mark=none,samples=100,domain=1.5:3] {-0.5*(x/1.5)^2-0.5};
			
			\addplot+[mark=none,domain=1.5:3,samples=100,%
              pattern=north east lines,%
              draw = black,
              pattern color=black]%
              {0.5*(x/1.5)^2+0.5}
              \closedcycle;
            \addplot+[mark=none,domain=1.5:3,samples=100,%
              pattern=north east lines,%
              draw = black,
              pattern color=black]%
              {-0.5*(x/1.5)^2-0.5}
              \closedcycle;
              
            \addplot[mark=none,samples=100,domain=-3:-1.5] {0.5*(x/1.5)^2+0.5};
			\addplot[mark=none,samples=100,domain=-3:-1.5] {-0.5*(x/1.5)^2-0.5};
			
			\addplot+[mark=none,domain=-3:-1.5,samples=100,%
              pattern=north east lines,%
              draw = black,
              pattern color=black]%
              {0.5*(x/1.5)^2+0.5}
              \closedcycle;
            \addplot+[mark=none,domain=-3:-1.5,samples=100,%
              pattern=north east lines,%
              draw = black,
              pattern color=black]%
              {-0.5*(x/1.5)^2-0.5}
              \closedcycle;
		\end{axis}
	\end{tikzpicture}
	\hfill
	\begin{tikzpicture}
		\begin{axis}[
				xmin = -3,
				xmax = 3,
				ymin = -3,
				ymax = 3,
				xticklabels = \empty,
				yticklabels = \empty,
				xlabel = $\Re$,
				ylabel = $\Im$,
				x label style={anchor=south west},
				axis x line = center,
				axis y line = center,
			]
			\addplot[mark=x,only marks, mark size = 5pt, mark options={line width = 1pt}] coordinates {(0,0)};
			\addplot[mark=o,only marks, mark size = 5pt, mark options={line width = 1pt}] coordinates {(0,0)};
			
			\node at (axis cs:-2,2){(d)};
			
			\draw (axis cs:1.5,1)--(axis cs:1.5,-1);
			\draw (axis cs:-1.5,1)--(axis cs:-1.5,-1);
			
			\draw[stealth-stealth] (axis cs: 0,0.75) -- (axis cs: 1.5,0.75) node[midway,above]{$\curlO(1)$};
			
			\addplot[mark=none,samples=100,domain=1.5:3] {0.5*(x/1.5)^2+0.5};
			\addplot[mark=none,samples=100,domain=1.5:3] {-0.5*(x/1.5)^2-0.5};
			
			\addplot+[mark=none,domain=1.5:3,samples=100,%
              pattern=north east lines,%
              draw = black,
              pattern color=black]%
              {0.5*(x/1.5)^2+0.5}
              \closedcycle;
            \addplot+[mark=none,domain=1.5:3,samples=100,%
              pattern=north east lines,%
              draw = black,
              pattern color=black]%
              {-0.5*(x/1.5)^2-0.5}
              \closedcycle;
              
            \addplot[mark=none,samples=100,domain=-3:-1.5] {0.5*(x/1.5)^2+0.5};
			\addplot[mark=none,samples=100,domain=-3:-1.5] {-0.5*(x/1.5)^2-0.5};
			
			\addplot+[mark=none,domain=-3:-1.5,samples=100,%
              pattern=north east lines,%
              draw = black,
              pattern color=black]%
              {0.5*(x/1.5)^2+0.5}
              \closedcycle;
            \addplot+[mark=none,domain=-3:-1.5,samples=100,%
              pattern=north east lines,%
              draw = black,
              pattern color=black]%
              {-0.5*(x/1.5)^2-0.5}
              \closedcycle;
		\end{axis}
	\end{tikzpicture}
	\caption{Depiction of the spectral results of Lemma \ref{lem:specLSH} and \ref{lem:specLcon} at $\varepsilon = 0$. (a) shows the spectrum of $L^\text{SH}$ if $c\neq 3c_u$ with a double central eigenvalue at zero. (b) shows the spectrum of $L^\text{SH}$ if $c = 3c_u$ with four eigenvalues at zero (two from $L^\text{SH}_{1}$ and two from $L^\text{SH}_{-1}$). (c) shows the spectrum of $L^\text{con}$ if $c \neq -c_v$ with a simple eigenvalue at zero. (d) shows the spectrum of $L^\text{con}$ if $c = -c_v$ with a double eigenvalue at zero. In each case the hyperbolic part of the spectrum is contained in the hatched area, which is $\curlO(1)$ bounded away from the imaginary axis.}
	\label{fig:spectralSituations}
\end{figure}

\begin{lemma}\label{lem:specLSH}
	Let $c\vert_{\varepsilon = 0} \neq c_u$.
	Then, there exists an $\varepsilon_0 > 0$ such that for all $\varepsilon \in (0,\varepsilon_0)$ the following holds.
	\begin{itemize}
		\item If $c\vert_{\varepsilon = 0} \neq 3c_u$, the matrix $L_n^\text{SH}$ for $n = \pm 1$ has a simple, central eigenvalue $\lambda_{c,\pm 1}$, which vanishes for $\varepsilon \rightarrow 0$.
		\item If $c = 3c_u + \curlO(\varepsilon)$, the matrix $L_n^\text{SH}$ for $n = \pm 1$ has two central eigenvalues $\lambda_{c,\pm 1}^{1/2}$, which vanish for $\varepsilon \rightarrow 0$.
	\end{itemize}
	All other eigenvalues of the $L_n^\text{SH}$ are bounded away from the imaginary axis uniformly in $\varepsilon$ and $n \in \Z$.
\end{lemma}
\begin{proof}
	First, we prove the statement for $\varepsilon = 0$ and then show that it persists for $\varepsilon > 0$ small.
	For $\varepsilon = 0$ it holds
	\begin{align*}
		\operatorname{det}(\lambda - L_n^\text{SH}) &= 1 - ic_u n -2n^2 + ic_u n^3 + n^4 + \lambda (-c + 4in + 3c_u n^2 - 4in^3) \\
		&\qquad\qquad+ \lambda^2 (2-3ic_u n - 6n^2) + \lambda^3(-c_u + 4in) + \lambda^4,
	\end{align*}
	where we used that $\cphase = c_u + \curlO(\varepsilon^2)$.
	For $n \in \Z$ with $\snorm{n}$ large, the four roots of the characteristic polynomial have the expansion
	\begin{align}
		\lambda = \begin{dcases}
			-in \pm \snorm{n}^{1/4}(-1)^{3/8} (c-c_u)^{1/4} + \curlO(\snorm{n}^\gamma) \\
			-in \pm \snorm{n}^{1/4}(-1)^{7/8} (c-c_u)^{1/4} + \curlO(\snorm{n}^\gamma)
		\end{dcases},
		\label{eq:asymptoticExpansionSH}
	\end{align}
	with $\gamma < 1/4$.
	Hence, the real part of $\lambda$ grows like $\snorm{n}^{1/4}$ for $\snorm{n} \rightarrow \infty$ since $c\vert_{\varepsilon = 0} \neq c_u$ by assumption.
	
	Next, we examine the existence of purely imaginary eigenvalues. We show that $L_n^\text{SH}$ has no purely imaginary eigenvalues if $n \neq \pm 1$.
	Additionally, we find that $L_n^\text{SH}$ with $n = \pm 1$ has a double eigenvalue on the imaginary axis if $c = 3c_u$ and a simple, purely imaginary eigenvalue else.
	Any purely imaginary eigenvalue $i\lambda_i$ with $\lambda_i \in \R$ satisfies
	\begin{align*}
		\Re(\det(i\lambda_i - L_n^\text{SH})) &= (n^2 - 1)^2 + (-4n + 4n^3)\lambda_i - \lambda_i^2 (2-6n^2) + 4n\lambda_i^3 + \lambda_i^4 \\
		&= (-1 + (n+\lambda_i)^2)^2 = 0
	\end{align*}
	and therefore $\lambda_i = -n \pm 1$, which are double roots.
	Computing the imaginary part of the characteristic polynomial yields
	\begin{align*}
		\Im(\det(i\lambda_i - L_n^\text{SH})) = nc_u (n^2-1) + \lambda_i (-c+3c_u n^2) + \lambda_i^2 3c_u n + \lambda^3 c_u.
	\end{align*}
	If $n = \pm 1$ this equation has a simple root at $\lambda_i = 0$ if $3c_u = c$ and a double root at $\lambda_i = 0$ if $3c_u = c$.
	Otherwise, inserting $\lambda_i = -n \pm 1$ yields $\Im(\det(i(-n\pm 1) - L_n^\text{SH})) = (c_u - c) (-n\pm 1)$.
	This is non-zero for all $n \neq \pm 1$ since $c\vert_{\varepsilon = 0} \neq c_u$ by assumption.
	
	It remains to prove that this setting persists for $\varepsilon > 0$ sufficiently small.
	Therefore, we note that that the eigenvalues of a matrix depend continously on its entries.
	Thus, for any $N \in \N$ there exists a $\varepsilon_0 > 0$ such that for all $\varepsilon \in (0,\varepsilon_0)$ and $\snorm{n} < N$ the spectrum of $L_n^\text{SH}$ is bounded away from the imaginary axis except for the central eigenvalues of $L_{\pm 1}^\text{SH}$, which stay close (with respect to $\varepsilon$) to the imaginary axis.
	Additionally, the coefficients of the asymptotic expansion in $n$, see \eqref{eq:asymptoticExpansionSH}, depend continuously on $\varepsilon$.
	Thus, for $n$ large, the real part of any eigenvalue of $L_n^\text{SH}$ grows like $\snorm{n}^{1/4}$.
	Therefore, all eigenvalues of $L_n^\text{SH}$ with $\snorm{n} > N$ can be uniformly bounded away from the imaginary axis by choosing $N$ sufficiently large.
	Hence, the number of eigenvalues close to the imaginary axis does not change for $\varepsilon \in (0,\varepsilon_0)$, which proves the lemma.
\end{proof}

\begin{lemma}\label{lem:specLcon}
	Let $c\vert_{\varepsilon = 0} \neq c_u$.
	Then, there exists an $\varepsilon_0 > 0$ such that for all $\varepsilon \in (0,\varepsilon_0)$ the following holds.
	\begin{itemize}
		\item If $c\vert_{\varepsilon = 0} \neq -c_v$, the matrix $L_0^\text{con}$ has one simple eigenvalue $\lambda = 0$.
		\item If $c = -c_v + \curlO(\varepsilon^m)$, $m \in \N$, the matrix $L_0^\text{con}$ has one simple eigenvalue $\lambda_0 = 0$ and one simple eigenvalue $\lambda_1 = -c-c_v = \curlO(\varepsilon^m)$.
	\end{itemize}
	All other eigenvalues of the $L_n^\text{con}$ are bounded away from the imaginary axis uniformly in $\varepsilon$ and $n \in \Z$.
\end{lemma}
\begin{proof}
	We proceed similar to the proof of Lemma \ref{lem:specLSH}.
	Setting $\varepsilon = 0$ we obtain
	\begin{align*}
		\operatorname{det}(\lambda - L_n^\text{con}) = i(c_u + c_v) n - n^2 + (2in + c + c_v)\lambda + \lambda^2,
	\end{align*}
	where we again used that $\cphase = c_u + \curlO(\varepsilon^2)$.
	Thus, for $n \in \Z$ with $\snorm{n}$ sufficiently large, any root of the characteristic polynomial satisfies the expansion
	\begin{align*}
		\lambda = -in \pm \snorm{n}^{1/2} (-1)^{1/4} \sqrt{c-c_u} + \curlO(\snorm{n}^\gamma)
	\end{align*}
	with $\gamma < 1/2$.
	Since $c\vert_{\varepsilon = 0} \neq c_u$ by assumption, it holds that $\Re(\lambda) = \curlO(\sqrt{\snorm{n}})$ for $n \rightarrow \infty$.
	
	Next, we show that $L_n^\text{con}$ has no imaginary eigenvalues for $n \neq 0$.
	Therefore, let $\lambda = i\lambda_i$ with $\lambda_i \in \R$ and solve $\Re(\det(i\lambda_i - L_n^\text{con})) = 0$, which yields that $\lambda_i = -n$.
	Inserting this into the imaginary part, we obtain
	\begin{align*}
		\Im(\det(-in - L_n^\text{con})) = (c_u - c)n,
	\end{align*}
	which does not vanish except for $n = 0$ since $c \neq c_u$ by assumption.
	Hence, $L_n^\text{con}$ has no purely imaginary eigenvalues for $n \neq 0$.
	For $n = 0$, we find the eigenvalues explicitly as $\lambda_0 = 0$ and $\lambda_0 = -c-c_v$, which proves the statement for $\varepsilon = 0$.
	
	Finally, the persistence can be proven similarly to Lemma \ref{lem:specLSH}.
	Therefore, we find that for $\varepsilon > 0$ sufficiently small the number of eigenvalues, which are $\varepsilon$-close to the imaginary axis does not change.
	Furthermore, we can again explicitly calculate the spectrum of $L_0^\text{con}$, which is given by a simple eigenvalue $\lambda_0 = 0$ and a simple eigenvalue $\lambda_0 = -c-c_v$.
	This completes the proof.
\end{proof}

\subsection{Center manifold result}
We now state a center manifold result for the spatial system \eqref{eq:spatDynFinal}.
Recall that the spectrum of the linear part $L$ in \eqref{eq:spatDynFinal} is the union of the eigenvalues of $L_n^\text{SH}$ and $L_n^\text{con}$.
Therefore, by Lemmas \ref{lem:specLSH} and \ref{lem:specLcon} the spectrum of $L$ splits into a finite set of central eigenvalues $\sigma_c$ containing the eigenvalues of $L_n^\text{SH}$ and $L_n^\text{con}$, which are $\varepsilon$-close to the imaginary axis and the remaining set of hyperbolic eigenvalues $\sigma_h$.
In particular, there is an $\varepsilon$-independent spectral gap around the imaginary axis, which does not contain any hyperbolic eigenvalues.
Hence, we can define the spectral projection onto the central eigenspace $\curlP_c$ by the Dunford integral formula
\begin{align*}
	\curlP_c = \dfrac{1}{2\pi i} \int_{\Gamma_c} (\lambda - L)^{-1} \,d\lambda,
\end{align*}
where $\Gamma_c$ is a simple, positively oriented Jordan curve surrounding precisely the central eigenvalues in $\sigma_c$.
Additionally, we define the spectral projection onto the hyperbolic eigenspace by $\curlP_h = I - \curlP_c$.
Finally, the linear operator $L$ is given as the direct sum of infinitely many finite dimensional matrices.
Hence, standard center manifold applies, see e.g.~\cite{haragusIooss11,schneiderUecker17} and the following result holds.

\begin{lemma}\label{lem:centerManifold}
	Let $\ell > 1/2$ and $c\vert_{\varepsilon = 0} \neq c_u$.
	Then, there exists an $\varepsilon_0 > 0$ such that for all $\varepsilon \in (0,\varepsilon_0)$ exists a neighborhood $O_c \subset \curlE_c := \curlP_c \curlE_\ell$ of the origin and a mapping $h = (h_u,h_v) : O_c \rightarrow \curlE_h := \curlP_h \curlE_\ell$ satisfying
	\begin{align}
		h(0) = 0 \text{ and } Dh(0) = 0
		\label{eq:propertyHCM}
	\end{align}
	such that the center manifold
	\begin{align*}
		\curlM_c := \{(U,V) = (U_c,V_c) + h(U_c,V_c) \,:\, (U_c,V_c) \in O_c\}
	\end{align*}
	is invariant and contains all small bounded solutions of \eqref{eq:spatDynFinal}.
	Furthermore, every solution of the reduced system
	\begin{align}
		\partial_\xi \begin{pmatrix}
			U_c \\ V_c
		\end{pmatrix} = L_c \begin{pmatrix}
			U_c \\ V_c
		\end{pmatrix} + \curlP_c \curlN(U_c + h_u(U_c,V_c), V_c + h_v(U_c,V_c)),
		\label{eq:redEquationCM}
	\end{align}
	where $L_c$ is the restriction of $L$ to $\curlE_c$, gives a solution to the full system \eqref{eq:spatDynFinal} via $(U,V) = (U_c,V_c) + h(U_c,V_c)$.
\end{lemma}

\section{Reduced equations and heteroclinic connections}\label{sec:redEqAndHetConnections}

In this section, we analyze the dynamics of the spatial system \eqref{eq:spatDynFinal} on the center manifold provided by Lemma \ref{lem:centerManifold}.
Therefore, we assume that $c\vert_{\varepsilon =0} \neq c_u$ throughout this section to ensure that the conditions of the center manifold result Lemma \ref{lem:centerManifold} are satisfied, see also Remark \ref{rem:speedCloseToPhaseVel}.
As discussed in the introduction, the analysis is split up into different parameter regimes, which we recapitulate here for convenience, see also Figure \ref{fig:Scenarios}.
\begin{enumerate}[label=\textbf{Scenario \Roman*:},align=left,ref=\Roman*]
	\item $c\vert_{\varepsilon = 0} \neq -c_v$ and $c\vert_{\varepsilon = 0} \neq 3c_u$ \label{scenario1}
	\item $c\vert_{\varepsilon = 0} \neq -c_v$ and $c = 3c_u + \varepsilon c_0$ with $c_0 \neq 0$ \label{scenario2}
	\item $c = -c_v + \varepsilon^2 c_0$ with $c_0 \neq 0$, $c\vert_{\varepsilon = 0} \neq 3c_u$ and $\gamma_2 = 0$ \label{scenario3}
	\item $c = -c_v + \varepsilon c_0$ with $c_0 \neq 0$, $c\vert_{\varepsilon = 0} \neq 3c_u$ and $\gamma_2 = \varepsilon\gamma_2^0$ \label{scenario4}
	\item $c = -c_v + \varepsilon c_0$ with $c_0 \neq 0$, $c\vert_{\varepsilon = 0} = 3c_u$ and $\gamma_2 = 0$ \label{scenario5}
\end{enumerate}
In every scenario we first derive the reduced equation on the center manifold.
Following, we establish the existence of heteroclinic solutions neglecting higher order terms, either using analytical (Scenarios \ref{scenario1}, \ref{scenario3} and \ref{scenario4}) or numerical methods (Scenarios \ref{scenario2} and \ref{scenario5}).
Afterwards, we show the persistence of the heteroclinic solutions with respect to small higher order terms, which establishes the existence of modulating traveling fronts by reverting the center manifold reduction in Lemma \ref{lem:centerManifold}.

\subsection{Scenario \ref{scenario1}}\label{sec:S1}
We consider the case that the modulating front spreads with speed $c$ such that $c\vert_{\varepsilon=0} \neq 3c_u$ as well as $c\vert_{\varepsilon = 0} \neq -c_v$.
That is, the speed of the front is away from the linear group velocities of both $u$ and $v$.
According to Lemmas \ref{lem:specLSH} and \ref{lem:specLcon}, the central spectrum then contains a simple eigenvalue $\lambda_0^\text{con} = 0$ coming from the conservation law and a pair of simple eigenvalues $\lambda_{c,\pm 1}^\text{SH}$ coming from $L_{\pm 1}^\text{SH}$.
Since $L_1^\text{SH} = \overline{L_{-1}^\text{SH}}$ it is sufficient to calculate an approximation of $\lambda_{c,1}^\text{SH}$, the central eigenvalue of $L_1^\text{SH}$.
The characteristic polynomial of $L_1^\text{SH}$ is given by
\begin{align}
	\det(\lambda - L_1^\text{SH}) = -\lambda^4 + \lambda^3(c_u-4i) + \lambda^2 (3ic_u+4) + \lambda(-3c_u+c) + \varepsilon^2 (\alpha_0 + i\omega_0^\ast),
	\label{eq:charPolS1}
\end{align}
utilizing that the phase velocity $\cphase$ of the periodic traveling wave solution $(\utw,\vtw)$ is given in Lemma \ref{lem:travelingWaves} by $\cphase = c_u + \varepsilon^2 \omega_0^\ast$.
Setting $\lambda = \varepsilon^2 \delta^\ast$ in \eqref{eq:charPolS1} and equating the $\varepsilon^2$-contribution to zero, we find that
\begin{align*}
	\delta^\ast = \dfrac{\alpha_0 + i\omega_0^\ast}{3c_u -c}.
\end{align*}
An application of the implicit function theorem (e.g.~similar to \cite[Appendix A]{hilder20}) then gives the expansion
\begin{align*}
	\lambda_{c,1}^\text{SH} = \varepsilon^2 \dfrac{\alpha + i\omega_0^\ast}{3c_u - c} + \curlO(\varepsilon^3),
\end{align*}
which holds for all $\varepsilon > 0$ sufficiently small.

\subsubsection{Derivation of the reduced equation}

We derive the reduced equation on the center manifold similar to \cite{eckmannWayne91,hilder20}.
Therefore, let $\phi_c^\text{SH}$ be the eigenvector corresponding to $\lambda_{c,1}^\text{SH}$, which is normalized such that its first component is equal to one.
Similarly, let $\phi_c^\text{con}$ be the eigenvector corresponding to $\lambda_c^\text{con} = 0$, which is normalized in the same fashion.
Then, we write
\begin{subequations}
	\begin{align}
		U_c(\xi) &= \varepsilon A(\varepsilon^2 \xi) \phi_c^\text{SH} e^{ip} + c.c., \label{eq:uCoordCMS1}\\
		V_c(\xi) &= \varepsilon^2 B(\varepsilon^2 \xi) \phi_c^\text{con}, \label{eq:vCoordCMS1}
	\end{align}
\end{subequations}
where $A(\tilde{\xi}) \in \C$, $B(\tilde{\xi}) \in \R$, $c.c.$ denotes the complex conjugated terms, and we additionally introduce the notation $X_c = (U_c,V_c)^T$.

\begin{remark}\label{rem:spatialRescaling}
	The spatial rescaling $\tilde{\xi} = \varepsilon^2 \xi$ is motivated by the $\varepsilon$-scaling of the central eigenvalue $\lambda_{c,1}^\text{SH} = \curlO(\varepsilon^2)$ and is used to balance the powers of $\varepsilon$ on the linear level of the reduced equation.
	More precisely, projecting the reduced equation on the center manifold \eqref{eq:redEquationCM} onto the eigenspaces spanned by $\phi_{c}^\text{SH}$ and $\phi_{c}^\text{con}$, respectively, leads to a system of the form
	\begin{align*}
		\varepsilon^3 \partial_{\tilde{\xi}} A &= \varepsilon \lambda_{c,1}^\text{SH} A + N_A(A,B,\varepsilon), \\
		\varepsilon^4 \partial_{\tilde{\xi}} B &= N_B(A,B,\varepsilon),
	\end{align*}
	where $N_A,N_B$ contain the nonlinearities.
	As we show below, the nonlinearities satisfy $N_A = \curlO(\varepsilon^3)$ and $N_B = \curlO(\varepsilon^4)$ and since $\lambda_{c,1}^\text{SH} = \curlO(\varepsilon^2)$, the powers of $\varepsilon$ balance in the leading order terms.
	A similarly motivated rescaling also reappears in the treatment of the other Scenarios \ref{scenario2}--\ref{scenario5}, see Sections \ref{sec:S2}--\ref{sec:S5}.
	However, we point out that it is not always possible to rescale the spatial variable $\xi$ such that all lowest order terms are independent of $\varepsilon$, see in particular the treatment of Scenario \ref{scenario4} in Section \ref{sec:S4}.
\end{remark}

Recall from Lemma \ref{lem:centerManifold} that functions on the center manifold are of the form
\begin{align*}
	U = U_c + h_u(X_c), \\
	V = V_c + h_v(X_c).
\end{align*}
Thus, we proceed by determining an approximation of the functions $h_u$ and $h_v$, respectively.
Following \cite[Corollary 2.12]{haragusIooss11} we have
\begin{align}
	Dh_u \partial_\xi U_c &= L_h^\text{SH} h_u + \curlP_h^\text{SH}(\curlN^\text{SH}(X_c + h(X_c))), \label{eq:equationForHuS1}\\
	Dh_v \partial_\xi V_c &= L_h^\text{con} h_v + \curlP_h^\text{con}(\curlN^\text{con}(X_c + h(X_c))), \label{eq:equationForHvS1}
\end{align}
where $\curlP_h^\text{SH}$ and $\curlP_h^\text{con}$ are the spectral projections onto the hyperbolic spectrum of $L^\text{SH}$ and $L^\text{con}$, respectively.
Furthermore, $L^\text{SH}$ and $L^\text{con}$ are defined by
\begin{align*}
	L^\text{SH} f = \sum_{n\in \Z} L_n^\text{SH} f_n e^{inp} \text{ and } L^\text{con}f = \sum_{n \in \Z} L_n^\text{con} f_n e^{inp},
\end{align*}
for $f \in \curlE_\ell$.
Using that $h_u$ and $h_v$ are at least quadratic in $\varepsilon$ since they depend at least quadratically on $X_c$, see \eqref{eq:propertyHCM}, we find
\begin{align*}
	\curlP_h \curlN^\text{SH}(X_c + h(X_c)) &= (0,0,0,2\varepsilon^2 i A^2 e^{2ip} + c.c.)^T + \curlO(\varepsilon^3), \\
	\curlP_h \curlN^\text{con}(X_c + h(X_c)) &= (0, 4\varepsilon^2 \gamma_1 A^2 e^{2ip} - 2i\gamma_2 \varepsilon^2 A^2 e^{2ip} + c.c.)^T + \curlO(\varepsilon^3).
\end{align*}
Thus, equating the lowest order terms in $\varepsilon$ of \eqref{eq:equationForHuS1} and \eqref{eq:equationForHvS1} to zero we then obtain
\begin{align*}
	0 &= L_h^\text{SH} h_u + (0,0,0,\varepsilon^2 i A^2 e^{2ip} + c.c.)^T, \\
	0 &= L_h^\text{con} h_v + (0,4\varepsilon^2 \gamma_1 A^2 e^{2ip} - 2i\gamma_2 \varepsilon^2 A^2 e^{2ip} + c.c.)^T,
\end{align*}
and therefore,
\begin{subequations}
	\begin{align}
		h_u &= \left(\varepsilon^2 \dfrac{iA^2e^{2ip}}{9+6ic_u} + c.c.,0,0,0\right)^T + \curlO(\varepsilon^3), \label{eq:huS1}\\
		h_v &= \left(\varepsilon^2 A^2 \dfrac{-2\gamma_1 + i\gamma_2}{2-i(c_u+c_v)} e^{2ip} + c.c.,0\right)^T + \curlO(\varepsilon^3). \label{eq:hvS1}
	\end{align}
\end{subequations}
Here, we used that $L_h^\text{SH}$ and $L_h^\text{con}$ have a bounded inverse since the hyperbolic eigenvalues are uniformly bounded away from the imaginary axis.

To obtain an evolution equation for $(A,B)$ on the center manifold, we additionally introduce the projections onto the eigenspaces spanned by $\phi_c^\text{SH}$ and $\phi_c^\text{con}$.
Following \cite{eckmannWayne91,hilder20} these projections are given by
\begin{align}
	\curlP_{\phi_c^\text{SH}}(U_1) = \dfrac{\scalarprod{\psi_c^\text{SH}}{U_1}}{\scalarprod{\psi_c^\text{SH}}{\phi_c^\text{SH}}}, \qquad \curlP_{\phi_c^\text{con}}(V_0) = \dfrac{\scalarprod{\psi_c^\text{con}}{V_0}}{\scalarprod{\psi_c^\text{con}}{\phi_c^\text{con}}},
	\label{eq:spectralProjections}
\end{align}
where $U_1$ is the first Fourier coefficient of $U$, $V_0$ is the zeroth Fourier coefficient of $V$ and $\psi_c^\text{SH}$ and $\psi_c^\text{con}$ are the eigenvectors of the adjoint matrices $(L_1^\text{SH})^\ast$ and $(L_0^\text{con})^\ast$ corresponding to the eigenvalues $\overline{\lambda_c^\text{SH}}$ and $\overline{\lambda_c^\text{con}}$, respectively.
Additionally, $\psi_c^\text{SH}$ and $\psi_c^\text{con}$ are normalized such that their last component is equal to one.
With this notation, we can write the reduced equation on the center manifold \eqref{eq:redEquationCM} equivalently as
\begin{align*}
	\varepsilon^3 \partial_{\tilde{\xi}} A &= \varepsilon\lambda_c^\text{SH} A + \curlP_{\phi_c^\text{SH}}(\curlN_1^\text{SH}(X_c + h(X_c))), \\
	\varepsilon^4 \partial_{\tilde{\xi}} B &= \curlP_{\phi_c^\text{con}}(\curlN_0^\text{con}(X_c + h(X_c))),
\end{align*}
where $\tilde{\xi} = \varepsilon^2\xi$ and we used the ansatz \eqref{eq:uCoordCMS1}--\eqref{eq:vCoordCMS1} and exploited that $\lambda_c^\text{con} = 0$.
Inserting the ansatz \eqref{eq:uCoordCMS1}--\eqref{eq:vCoordCMS1} and the expansion for $(h_u,h_v)$ from \eqref{eq:huS1}--\eqref{eq:hvS1} into the nonlinearities \eqref{eq:nonlinSH}--\eqref{eq:nonlinCon} then yields
\begin{align*}
	\curlP_{\phi_c^\text{SH}}(\curlN_1^\text{SH}(X_c + h(X_c))) &= \dfrac{\varepsilon^3}{\scalarprod{\psi_c^\text{SH}}{\phi_c^\text{SH}}}\left[A\left(B + \dfrac{-2\gamma_1 + i\gamma_2}{2-i(c_u+c_v)} \snorm{A}^2\right) - \dfrac{1}{9+6ic_u}A\snorm{A}^2 - 3 A\snorm{A}^2\right], \\
	&\qquad\qquad +\curlO(\varepsilon^4), \\
	\curlP_{\phi_c^\text{con}}(\curlN_0^\text{con}(X_c + h(X_c))) &= \dfrac{2\gamma_2\varepsilon^4}{\scalarprod{\psi_c^\text{con}}{\phi_c^\text{con}}} \partial_{\tilde{\xi}} \snorm{A}^2 + \curlO(\varepsilon^5).
\end{align*}
To obtain the latter expression, we used that
\begin{align*}
	\sum_{k+j = 0} U_{k0}U_{j1} = \dfrac{1}{2}\partial_\xi \sum_{k+j=0} U_{k0}U_{j0},
\end{align*}
since $U_{j1} = \partial_\xi U_{j0}$.
Finally, we calculate using \cite[Appendix C]{eckmannWayne91} that
\begin{align*}
	\scalarprod{\psi_c^\text{SH}}{\phi_c^\text{SH}} &= -\partial_\lambda \det(\lambda - L_1^\text{SH})\vert_{\lambda = \lambda_{c,1}^\text{SH}} = 3c_u - c + \curlO(\varepsilon^2), \\
	\scalarprod{\psi_c^\text{con}}{\phi_c^\text{con}} &= -(c+c_v).
\end{align*}
Hence, to lowest order in $\varepsilon$, the dynamical behavior on the center manifold is determined by
\begin{align*}
	\partial_{\tilde{\xi}} A &= \dfrac{\alpha_0 + i\omega_0^\ast}{3c_u - c} A + \dfrac{1}{3c_u - c}\left(AB + \left(-3-\dfrac{1}{9+6ic_u} + \dfrac{-2\gamma_1+i\gamma_2}{2-i(c_u+c_v)}\right) A\snorm{A}^2\right), \\
	\partial_{\tilde{\xi}} B &= \dfrac{2\gamma_2}{c_v+c} \partial_{\tilde{\xi}}(\snorm{A}^2).
\end{align*}
Integrating the second equation with respect to $\tilde{\xi}$ we find that $B$ is to lowest order determined by $\snorm{A}^2$.
Therefore, we arrive at
\begin{align}
	\partial_{\tilde{\xi}} A &= \dfrac{\alpha_0 + i \omega_0^\ast}{3c_u - c} A + \dfrac{A\snorm{A}^2}{3c_u -c} \left(\dfrac{2\gamma_2}{c_v+c} - 3 - \dfrac{1}{9+6ic_u} + \dfrac{-2\gamma_1 + i\gamma_2}{2-i(c_u+c_v)}\right).
	\label{eq:redEqS1}
\end{align}

\begin{remark}\label{rem:nonlinDispersionMotivationS1}
	We point out that the lowest order connection of $B$ and $A$ is only present due to the nonlinear dispersion term $\gamma_2 \partial_x(u^2)$ in the conservation law \eqref{eq:ConDis}.
	In particular, if the nonlinear dispersion term in \eqref{eq:ConDis} is removed (i.e.~$\gamma_2 = 0$), $B$ is constant to lowest order.
	Mathematically, this motivates the inclusion of the nonlinear dispersion in the conservation law \eqref{eq:ConDis} to enrich the dynamics of the system.
	Furthermore, a similar motivation for the presence of the nonlinear dispersion can be found in Scenario \ref{scenario4} treated in Section \ref{sec:S4}.
\end{remark}

\subsubsection{Existence of heteroclinic orbits}

We now analyze the dynamics of the reduced equation \eqref{eq:redEqS1}.
As in the construction of periodic solutions in Lemma \ref{lem:travelingWaves} we first set $\gamma_1 = \gamma_2 = 0$.
Thus, \eqref{eq:redEqS1} simplifies to
\begin{align*}
	\partial_{\tilde{\xi}} A = \dfrac{\alpha_0 + i \omega_0^\ast}{3c_u - c} A - \dfrac{A\snorm{A}^2}{3c_u - c}\left(3 + \dfrac{1}{9+6ic_u}\right).
\end{align*}
Following the proof of Lemma \ref{lem:travelingWaves}, this equation has a circle of non-trivial fixed points $A^\ast$ with
\begin{align*}
	\snorm{A^\ast}^2 = \dfrac{(9+4c_u^2)\alpha_0}{4(7+3c_u^2)}.
\end{align*}
We now show the existence of heterclinic orbits connecting the circle of non-trivial fixed points to the origin.
To do this, we use the invariance with respect to $A \mapsto Ae^{i\phi}$ and write $A$ in polar coordinates, i.e.~$A = r_A e^{i\phi_A}$, which yields
\begin{subequations}
	\begin{align}
		\partial_{\tilde{\xi}} r_A &= \dfrac{\alpha_0}{3c_u - c} r_A - \dfrac{r_A^3}{3c_u -c}\left(3 + \dfrac{1}{9+4c_u^2}\right), \label{eq:redEqS1PolCoordR} \\
		\partial_{\tilde{\xi}} \phi_A &= \dfrac{\omega_0^\ast}{3c_u - c} + \dfrac{r_A^2}{3c_u -c}\dfrac{2c_u}{27+12c_u^2}. \label{eq:redEqS1PolCoordPhi}
	\end{align}
\end{subequations}
Since \eqref{eq:redEqS1PolCoordR} is independent of the angle $\phi_A$, it is sufficient to construct heteroclinic connections in the $r_A$-equation.
Linearizing \eqref{eq:redEqS1PolCoordR} about $r_A = 0$ yields
\begin{align*}
	L_0 = \dfrac{\alpha_0}{3c_u -c}
\end{align*}
and thus, the origin is stable for $c > 3c_u$ and unstable for $c < 3c_u$.
Similarly, the linearization about $r_A = A^\ast$ is given by
\begin{align*}
	L_{A^\ast} = -\dfrac{4\alpha_0}{3c_u - c}.
\end{align*}
Hence $r_A = A^\ast$ is unstable for $c > 3c_u$ and stable for $c < 3c_u$.
Since \eqref{eq:redEqS1PolCoordR} does not possess a stationary state in $(0,A^\ast)$, there exists a heteroclinic solution $r_{c,\text{het}}$ of \eqref{eq:redEqS1PolCoordR} such that the following holds.
\begin{itemize}
	\item If $c > 3c_u$, then
		\begin{align*}
			\lim_{\tilde{\xi} \rightarrow -\infty} r_{c,\text{het}}(\tilde{\xi}) = A^\ast \text{ and } \lim_{\tilde{\xi} \rightarrow +\infty} r_{c,\text{het}}(\tilde{\xi}) = 0.
		\end{align*}
	\item If $c < 3c_u$, then
		\begin{align*}
			\lim_{\tilde{\xi} \rightarrow -\infty} r_{c,\text{het}}(\tilde{\xi}) = 0 \text{ and } \lim_{\tilde{\xi} \rightarrow +\infty} r_{c,\text{het}}(\tilde{\xi}) = A^\ast.
		\end{align*}
\end{itemize}

\begin{remark}
	We highlight that the appearance of $3c_u$ as the switching point is not a coincidence since we expect that the group velocity of the periodic traveling waves, see Lemma \ref{lem:travelingWaves}, is close to $3c_u$ for $\varepsilon > 0$ sufficiently small.
	Here, we recall that the instability leading to the bifurcation of the periodic traveling waves is driven by the unstable spectrum in the dispersive Swift-Hohenberg equation \eqref{eq:SHDis} at Fourier wave numbers close to $k_c = \pm 1$.
	This leads to the above expectation that the nonlinear group velocity is close to the linear group velocity of the dispersive Swift-Hohenberg equation $c_\text{lin,phase,u}(k_c) = 3c_u$.
\end{remark}

\subsubsection{Persistence of heteroclinic orbits and existence of modulating fronts}

It remains to show that the heteroclinic connection $r_{c,\text{het}}$ persists if higher order pertubrations in $(\varepsilon, \gamma_1,\gamma_2)$ are introduced to the system.
Using the conservation law structure of \eqref{eq:ConDis} the full reduced equation for $B$ on the center manifold, including higher order terms, is of the form
\begin{align*}
	\partial_{\tilde{\xi}} B &= \dfrac{2\gamma_2}{c+c_v} \partial_{\tilde{\xi}} \snorm{A}^2 + \partial_{\tilde{\xi}} \tilde{g}(\varepsilon, \gamma_1, \gamma_2, A,B)
\end{align*}
with $\tilde{g} \in \curlO(\varepsilon)$.
An application of the implicit function theorem thus yields
\begin{align*}
	B = \dfrac{2\gamma_2}{c+c_v} \snorm{A}^2 + \curlO(\varepsilon(1+\snorm{\gamma_1}+\snorm{\gamma_2})\snorm{A}^2).
\end{align*}
Hence, $B$ is still determined by $\snorm{A}^2$ when higher order terms are included.
Next, using the translation invariance of \eqref{eq:SHDis}--\eqref{eq:ConDis} as in the proof of Lemma \ref{lem:travelingWaves}, we obtain the persistence of the non-trivial stationary state $r_A = A^\ast$ for the reduced dynamics.
Then, continuous dependence of $\lambda_c^\text{SH}$, $\phi_c^\text{SH}$ and $\psi_c^\text{SH}$ on $(\varepsilon,\gamma_1,\gamma_2)$ proves the persistence of $r_{c,\text{het}}$ under higher order perturbations on the center manifold.
Finally reversing the center manifold reduction in Lemma \ref{lem:centerManifold}, we obtain the existence of modulating traveling fronts as stated in the following theorem.

\begin{theorem}\label{thm:S1}
	Let $c\vert_{\varepsilon = 0} \notin \{c_u,-c_v,3c_u\}$.
	Then, there exist $\varepsilon_0 > 0$ and $\gamma^\ast > 0$ such that for all $\varepsilon \in (0,\varepsilon_0)$ and $\gamma_1,\gamma_2 \in (-\gamma^\ast,\gamma^\ast)$ the system \eqref{eq:SHDis}--\eqref{eq:ConDis} has the following solutions.
	\begin{enumerate}
		\item If $c > 3c_u$ there exists a family of modulating traveling front solutions $(u_\text{front},v_\text{front}) = (u_\text{front},v_\text{front})(\xi,p)$ satisfying \eqref{eq:modfrontBC}, which are of the form
		\begin{align*}
			u_\text{front}(\xi,p) &= \varepsilon 2\Re(A_\text{het}(\varepsilon^2\xi)) \cos(p) + \curlO(\varepsilon^2), \\
			v_\text{front}(\xi,p) &= \varepsilon^2 \left(\dfrac{2\gamma_2}{c+c_v} \snorm{A_\text{het}(\varepsilon\xi)}^2 + 2 \Re\left(\dfrac{-2\gamma_1 + i\gamma_2}{2 - i(c_u+c_v)}\Ahet(\varepsilon^2 \xi)^2\right) \cos(2p)\right) + \curlO(\varepsilon^3),
		\end{align*}
		where $A_\text{het}$ is a heteroclinic orbit of \eqref{eq:redEqS1} with 
		\begin{align*}
			\lim_{\tilde{\xi} \rightarrow -\infty} A_\text{het}(\tilde{\xi}) = A^\ast \text{ and } \lim_{\tilde{\xi} \rightarrow +\infty} A_\text{het}(\tilde{\xi}) = 0.
		\end{align*}
		\item If $c < 3c_u$ there exists a family of ``reversed'' modulating traveling front solutions $(u_\text{front},v_\text{front}) = (u_\text{front},v_\text{front})(\xi,p)$ satisfying \eqref{eq:revModfrontBC}, which are of the form
		\begin{align*}
			u_\text{front}(\xi,p) &= \varepsilon 2\Re(\check{A}_\text{het}(\varepsilon^2\xi)) \cos(p) + \curlO(\varepsilon^2), \\
			v_\text{front}(\xi,p) &= \varepsilon^2 \left(\dfrac{2\gamma_2}{c+c_v} \snorm{\check{A}_\text{het}(\varepsilon\xi)}^2 + 2 \Re\left(\dfrac{-2\gamma_1 + i\gamma_2}{2 - i(c_u+c_v)}\check{A}_\text{het}(\varepsilon^2 \xi)^2\right) \cos(2p)\right) + \curlO(\varepsilon^3),
		\end{align*}
		where $\check{A}_\text{het}$ is a heteroclinic orbit of \eqref{eq:redEqS1} with 
		\begin{align*}
			\lim_{\tilde{\xi} \rightarrow -\infty} \check{A}_\text{het}(\tilde{\xi}) = 0 \text{ and } \lim_{\tilde{\xi} \rightarrow +\infty} \check{A}_\text{het}(\tilde{\xi}) = A^\ast.
		\end{align*}	
	\end{enumerate}
	Here $\xi = x-ct$ and $p = x-\cphase t$ with the phase velocity $\cphase$ given in Lemma \ref{lem:travelingWaves}.
\end{theorem}

\subsection{Scenario \ref{scenario2}}\label{sec:S2}
In the second scenario, the spreading speed of the modulating front $c$ is assumed to be close to the linear group velocity of the bifurcating traveling wave, that is, $c = 3c_u + \varepsilon c_0$ with $c_0 \neq 0$.
Using Lemma \ref{lem:specLSH} the central spectrum of $L^\text{SH}$ therefore consists of the two central eigenvalues of $L_{\pm 1}^\text{SH}$, respectively.
Additionally, we assume that $c\vert_{\varepsilon = 0} \neq -c_v$ or equivalently $3c_u \neq -c_v$ and therefore, the conservation law part contributes only a simple zero eigenvalue to the central spectrum, see Lemma \ref{lem:specLcon}.

We now calculate the central eigenvalues of $L_1^\text{SH}$, which are denoted by $\lambda_{c,\pm}^\text{SH}$.
Proceeding similarly to Section \ref{sec:S1}, the characteristic polynomial of $L_1^\text{SH}$ is given by 
\begin{align*}
	\det(\lambda - L_1^\text{SH}) = -\lambda^4+\lambda^3(c_u-4i) + \lambda(3ic_u+4) + \lambda\varepsilon c_0 + \varepsilon^2(\alpha_0 + i\omega_0^\ast),
\end{align*}
where we used that $\cphase = c_u + \varepsilon^2 \omega_0^\ast$ and $c = 3c_u + \varepsilon c_0$.
Inserting the ansatz $\lambda = \varepsilon\delta$ and equating the $\varepsilon^2$-terms to zero then yields
\begin{align*}
	\delta_\pm = \dfrac{-c_0 \pm \sqrt{c_0^2-4(3ic_u + 4)(\alpha_0 + i\omega_0^\ast)}}{8+6ic_u} =: -\dfrac{c_0}{8+6ic_u} \pm \dfrac{\Delta}{8+6ic_u}.
\end{align*}
Applying the implicit function theorem then gives that
\begin{align*}
	\lambda_{c,\pm}^\text{SH} = \varepsilon \delta_\pm + \curlO(\varepsilon^2) = -\dfrac{c_0}{8+6ic_u} \pm \dfrac{\Delta}{8+6ic_u} + \curlO(\varepsilon^2).
\end{align*}
As in the first scenario it is sufficient to consider the eigenvalues of $L_1^\text{SH}$ since the central eigenvalues of $L_{-1}^\text{SH}$ are given by $\overline{\lambda_{c,\pm}^\text{SH}}$.

\subsubsection{Derivation of the reduced system}

After calculating the central eigenvalues, we proceed to derive the reduced equation on the center manifold.
Adapting the approach of the previous Section \ref{sec:S1}, we introduce the coordinates
\begin{subequations}
	\begin{align}
		U_c(\xi) &= \left(\varepsilon A_+(\varepsilon \xi) \phi_{c,+}^\text{SH} + \varepsilon A_-(\varepsilon\xi) \phi_{c,-}^\text{SH}\right) e^{ip} + c.c., \label{eq:uCoordCMS2}\\
		V_c(\xi) &= \varepsilon^2 B(\varepsilon \xi) \phi_c^\text{con}, \label{eq:vCoordCMS2}
	\end{align}
\end{subequations}
where $\phi_{c,\pm}^\text{SH}$ are the eigenvectors of $L_1^\text{SH}$ corresponding to the central eigenvalues $\lambda_{c,\pm}^\text{SH}$ and are normalized such that their first component is equal to one.
Similarly, $\phi_c^\text{con}$ is the eigenvector of $L_0^\text{con}$ corresponding to its zero eigenvalue.
We point out that the spatial rescaling $\hat{\xi} = \varepsilon \xi$ done in \eqref{eq:uCoordCMS2}--\eqref{eq:vCoordCMS2} is different from the rescaling in \eqref{eq:uCoordCMS1}--\eqref{eq:vCoordCMS1} since the central eigenvalues $\lambda_{c,\pm}^\text{SH}$ are of order $\varepsilon$ instead of $\varepsilon^2$.
Hence, to balance the $\varepsilon$-powers on the linear level of the reduced equation a different rescaling is necessary, see also Remark \ref{rem:spatialRescaling}.

Next, we recall that on the center manifold holds that
\begin{align*}
	U(\xi) &= U_c(\xi) + h_u(A_+,A_-,B), \\
	V(\xi) &= V_c(\xi) + h_v(A_+,A_-,B),
\end{align*}
with $h = (h_u,h_v)$ from Lemma \ref{lem:centerManifold}.
Redoing the computations done in the previous Section \ref{sec:S1} we find that
\begin{align*}
	h_u &= \left(\varepsilon^2 \dfrac{i(A_+ + A_-)}{9+6ic_u} e^{2ip} + c.c.,0,0,0\right)^T + \curlO(\varepsilon^3), \\
	h_v &= \left(\varepsilon^2 \dfrac{-2\gamma_1 + i\gamma_2}{2-i(c_u+c_v)}(A_+ + A_-) e^{2ip} + c.c., 0\right)^T + \curlO(\varepsilon^3).
\end{align*}
Additionally, let again $\psi_{c,\pm}^\text{SH}$ be the eigenvectors of the adjoint problem normalized such that their last component is one.
Then, it hold that
\begin{align*}
	\scalarprod{\psi_{c,\pm}^\text{SH}}{\phi_{c,\pm}^\text{SH}} = -\partial_\lambda \det(\lambda - L_1^\text{SH})\vert_{\lambda = \lambda_{c,\pm}^\text{SH}} = \mp \varepsilon \Delta + \curlO(\varepsilon^2).
\end{align*}
Proceeding further along the lines of Section \ref{sec:S1} and utilizing that the central spectrum of $L^\text{con}$ does not change, we find that $B$ is determined by $\snorm{A_+ + A_-}^2$ to lowest order in $\varepsilon$.
That is, neglecting higher order terms we obtain
\begin{align*}
	B = \dfrac{2\gamma_1}{c+c_v}\snorm{A_+ + A_-}^2.
\end{align*}
The equations for $A_+$, $A_-$ can then be obtained by projecting onto the eigenspaces spanned by $\phi_{c,\pm}^\text{SH}$ and are given by
\begin{align*}
	\partial_{\hat{\xi}} A_+ &= \delta_+ A_+ + a_\text{cub} (A_+ + A_-) \snorm{A_+ + A_-}^2, \\
	\partial_{\hat{\xi}} A_- &= \delta_- A_- - a_\text{cub} (A_+ + A_-) \snorm{A_+ + A_-}^2,
\end{align*}
up to higher order terms, with $\hat{\xi} = \varepsilon \xi$ and 
\begin{align*}
	a_\text{cub} = -\dfrac{1}{\Delta}\left(\dfrac{2\gamma_1}{c+c_v} - 3 - \dfrac{1}{9+6ic_u} + \dfrac{-2\gamma_1 + i\gamma_2}{2-i(c_u+c_v)}\right) =: -\dfrac{\tilde{a}_\text{cub}}{\Delta}.
\end{align*}
Finally, similar to \cite{hilder20} we substitute
\begin{align}
	A := A_+ + A_- \text{ and } \tilde{A} := -\dfrac{c_0}{8+6ic_u} (A_+ + A_-) + \dfrac{\Delta}{8+6ic_u} (A_+ - A_-),
	\label{eq:substitutionS2}
\end{align}
and arrive at the final equation on the center manifold
\begin{subequations}
	\begin{align}
		\partial_{\hat{\xi}} A &= \tilde{A}, \label{eq:redEqS2_1}\\
		\partial_{\hat{\xi}} \tilde{A} &= - \dfrac{2c_0}{8+6ic_u} \tilde{A} + \dfrac{\Delta^2 - c_0^2}{(8+6ic_u)^2} A - \dfrac{2 \tilde{a}_\text{cub}}{8+6ic_u} A \snorm{A}^2. \label{eq:redEqS2_2}
	\end{align}
\end{subequations}

\subsubsection{Numerical analysis of the reduced system}

We study the dynamics of the reduced system \eqref{eq:redEqS2_1}--\eqref{eq:redEqS2_2}.
Again, we first set $\gamma_1 = \gamma_2 = 0$ and study the persistence of the dynamical behavior afterwards.
Additionally, the analysis can be restricted to the case $c_0 \geq 0$ since the system \eqref{eq:redEqS2_1}--\eqref{eq:redEqS2_2} is reversible with respect to $R : (A, \tilde{A},c_0) \mapsto (A, -\tilde{A},-c_0)$.
That is, if $\hat{\xi} \mapsto (A(\hat{\xi}), \tilde{A}(\hat{\xi}))$ is a solution of \eqref{eq:redEqS2_1}--\eqref{eq:redEqS2_2} with $c_0$, then $\hat{\xi} \mapsto (A(-\hat{\xi}),-\tilde{A}(-\hat{\xi}))$ solves \eqref{eq:redEqS2_1}--\eqref{eq:redEqS2_2} with $-c_0$.

We are interested in the existence of heteroclinic orbits connecting the fixed point $(A^\ast,0)$ to the origin, where $A^\ast$ is given in Lemma \ref{lem:travelingWaves}.
Hence, we consider the stability of these fixed points.
Since any solution to \eqref{eq:redEqS2_1}--\eqref{eq:redEqS2_2} is in general complex valued, we split the solution into their real and imaginary part and write $A = A_r + iA_i$ and $\tilde{A} = \tilde{A}_r + i \tilde{A}_i$, where the subscripts $r$ and $i$ denote the real and imaginary parts, respectively.
Inserting this into \eqref{eq:redEqS2_1}--\eqref{eq:redEqS2_2} then yields the system
\begin{subequations}
	\begin{align}
		\partial_{\hat{\xi}} A_r &= \tilde{A}_r, \label{eq:redEqS2_Complex_1} \\
		\partial_{\hat{\xi}} A_i &= \tilde{A}_i, \label{eq:redEqS2_Complex_2}\\
		\partial_{\hat{\xi}} \tilde{A}_r &= a_r A_r - a_i A_i + b_r \tilde{A}_r - b_i \tilde{A}_i + c_r A_r(A_r^2 + A_i^2) - c_i A_i(A_r^2 + A_i^2), \label{eq:redEqS2_Complex_3}\\
		\partial_{\hat{\xi}} \tilde{A}_i &= a_r A_i + a_i A_r + b_r \tilde{A}_i + b_i \tilde{A}_r + c_r A_i(A_r^2 + A_i^2) + c_i A_r(A_r^2 + A_i^2), \label{eq:redEqS2_Complex_4}
	\end{align}
\end{subequations}
where the coefficients are given by
\begin{align*}
	a := \dfrac{\Delta^2 - c_0^2}{(8+6ic_u)^2}, \quad b:= -\dfrac{2c_0}{8+6ic_u}, \quad c := -\dfrac{2\tilde{a}_\text{cub}}{8+6ic_u}.
\end{align*}

The linearlization about the invading state $(A^\ast,0,0,0)$ is given by
\begin{align*}
	L_\text{invading} = \begin{pmatrix}
		0 & 0 & 1 & 0 \\
		0 & 0 & 0 & 1 \\
		a_r + 3c_r (A^\ast)^2 & -a_i - c_i(A^\ast)^2 & b_r & -b_i \\
		a_i + 3c_i (A^\ast)^2 & a_r + c_r (A^\ast)^2 & b_i & b_r
	\end{pmatrix} = \begin{pmatrix}
		0 & 0 & 1 & 0 \\
		0 & 0 & 0 & 1 \\
		2c_r (A^\ast)^2 & 0 & b_r & -b_i \\
		2c_i (A^\ast)^2 & 0 & b_i & b_r
	\end{pmatrix},
\end{align*}
where we used $A^\ast \neq 0$ and that $(A^\ast,0,0,0)$ is a fixed point of \eqref{eq:redEqS2_Complex_1}--\eqref{eq:redEqS2_Complex_4}.
Therefore, $L_\text{invading}$ has one zero eigenvalue.
Furthermore, numerical computations show that $L_\text{invading}$ additionally has one unstable and two stable eigenvalues, independently of the value of $c_0 > 0$, see Figure \ref{fig:invadingSpectrum}.

\begin{figure}
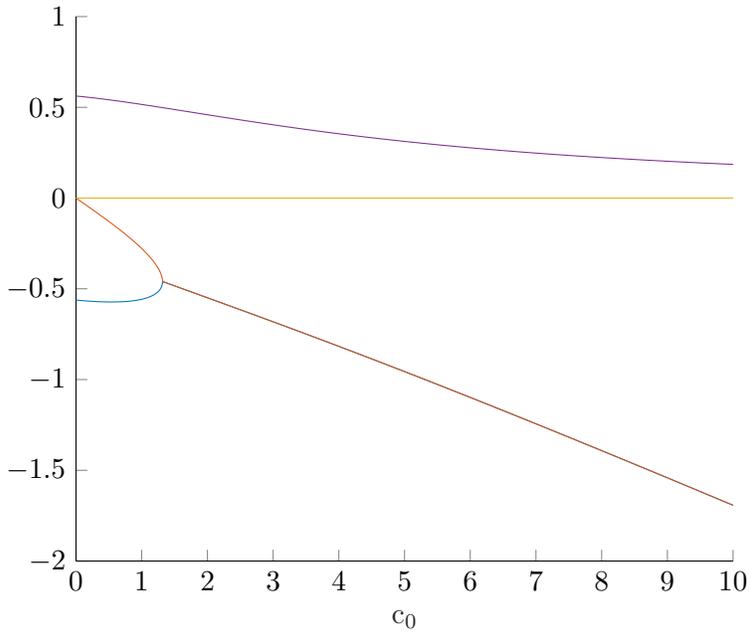

	\centering
	\include{Images/invadingSpectrum}
	\caption{Real part of the spectrum of $L_\text{invading}$ for $\alpha_0 = c_u = 1$ and $0 \leq c_0 \leq 10$.}
	\label{fig:invadingSpectrum}
\end{figure}

Next, we study the linearization about the origin, denoted by $L_0$.
Numerically, we find the following behavior.
There exists a $c_0^\ast > 0$ such that for all $c_0 > c_0^\ast$, the origin is a stable hyperbolic fixed point, that is, all eigenvalues of $L_0$ have negative real part.
At $c_0 = c_0^\ast$ a pair of complex conjugated eigenvalues crosses the imaginary axis, leading to a Hopf bifurcation.
Hence, periodic solutions bifurcate from the origin.
We study this bifurcation further using AUTO \cite{auto07p} and obtain the existence of a $c_0^{\ast\ast} \in (0,c_0^\ast)$ such that the periodic solutions are spectrally stable for $c_0 \in (c_0^{\ast\ast},c_0^\ast)$.
That is, with the exception of one simple Floquet multiplier (see e.g.~\cite{schneiderUecker17}) with absolute value one, all Floquet multipliers of the periodic solutions have absolute value less than one if $c_0 \in (c_0^{\ast\ast},c_0^\ast)$, see Figure \ref{fig:floquetMultipliers}.
Finally, at $c_0 = c_0^{\ast\ast}$ the family of periodic solutions destabilizes and undergoes a secondary Torus bifurcation.
The findings are summarized in the bifurcation diagram Figure \ref{fig:bifurcationDiagram}.
In the depicted case $\alpha_0 = c_u = 1$ the bifurcation points for $c_0$ are approximately given by $c_0^\ast \approx 1.55172$ and $c_0^{\ast\ast} \approx 1.01500$.

\begin{figure}
	\begin{minipage}{0.47\textwidth}
		\centering
		\include{Images/FloquetMultiplier/plotFloquetMultiplier}
		\caption{Absolute values of Floquet multipliers corresponding to the bifurcating periodic solution for $\alpha_0 = c_u = 1$.}
		\label{fig:floquetMultipliers}
	\end{minipage}
	\hfill
	\begin{minipage}{0.47\textwidth}
		\centering
		\include{Images/BifDiagram/bifDiagram}
		\caption{Bifurcation diagram of \eqref{eq:redEqS2_Complex_1}--\eqref{eq:redEqS2_Complex_4} for $\alpha_0 = c_u = 1$ starting from the origin. The squares mark the bifurcation points, where HB denotes the Hopf bifurcation and TB denotes the Torus bifurcation.}
		\label{fig:bifurcationDiagram}
	\end{minipage}
\end{figure}

Using the spectral information at the fixed points, we now numerically study the existence of heteroclinic orbits in \eqref{eq:redEqS2_Complex_1}--\eqref{eq:redEqS2_Complex_4}, which converge to $(A^\ast,0,0,0)$ as $\hat{\xi} \rightarrow -\infty$.
Since the invading state $(A^\ast,0,0,0)$ has a one-dimensional unstable manifold, we employ a shooting-type algorithm, that is, we choose an initial condition on the unstable manifold and compute the forward orbit.
In practice, the initial condition is chosen on the unstable eigenspace close to the fixed point, which exploits that the unstable eigenspace is tangential to the unstable manifold.
We find the following behavior.
For $c_0 > c_0^\ast$ the forward orbit converges to the origin, which can be expected since the origin is stable provided that $c_0 > c_0^\ast$.
Note that, in contrast to the non-dispersive setting in \cite{hilder20}, the corresponding front is in general not monotonically decreasing since the eigenvalues of $L_0$ are genuinely complex.
As $c_0$ decreases below $c_0^\ast$, the forward orbit connects to the bifurcating periodic solution instead of the origin.
Furthermore, for $c_0 < c_0^{\ast\ast}$ but close to $c_0^{\ast\ast}$ the forward orbit connects to a quasiperiodic solution, which is a result of the secondary bifurcation at $c_0^{\ast\ast}$.
Finally, further decreasing $c_0$ leads to instability and no heteroclinic connections are found in the numerics. 
The numerical simulations indicate that this break up of heteroclinic connections occurs when the amplitude of the quasiperiodic state exceeds the amplitude of the invading state. 
As a demonstration of this behavior, we plot $2\Re(A) = 2\Re(A_+ + A_-)$ for each case in Figure \ref{fig:hetConnectionsGroupVelocity}.
To relate this to the full modulating front, recall that
\begin{align*}
	u(t,x) = \curlU(\xi,p) = \varepsilon(A_+(\varepsilon\xi) + A_-(\varepsilon\xi)) e^{ip} + c.c. + \curlO(\varepsilon^2) = 2\varepsilon\Re(A(\varepsilon\xi))\cos(p) + \curlO(\varepsilon^2),
\end{align*}
where we used \eqref{eq:modfrontAnsatz}, \eqref{eq:uCoordCMS2} and the fact that the eigenvectors $\phi_{c,\pm}^\text{SH}$ are normalized such that their first component is equal to one.

\begin{figure}
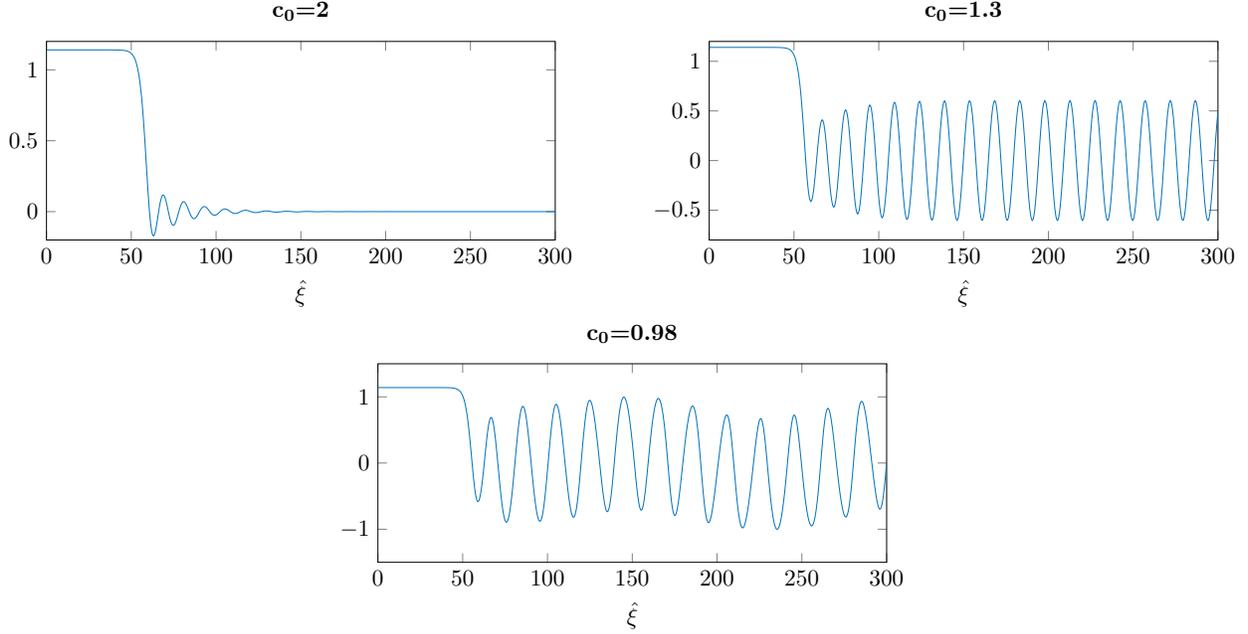

	\centering
	\include{Images/HetConnectionsGroupVel/plotHeteroclinicsComplete}
	\caption{Plot of amplitude of the modulating front $2\operatorname{Re}(A)$ for $\alpha_0 = c_u = 1$ and $c_0 = 2, 1.3, 0.98$. The heteroclinic connects the invading state $2A^\ast \approx 1.0954$ to the origin for $c_0 = 2$, to the bifurcating periodic state for $c_0 = 1.3$ and to the quasiperiodic state for $c_0 = 0.98$.}	\label{fig:hetConnectionsGroupVelocity}
\end{figure}

\subsubsection{Persistence of heteroclinic orbits and existence of modulating fronts}

We now turn to the persistence of heteroclinic connections, established numerically in the previous section, with respect to small perturbations in $(\varepsilon,\gamma_1,\gamma_2)$.
First, proceeding similarly to Section \ref{sec:S1}, we obtain that
\begin{align*}
	B = \dfrac{2\gamma_1}{c+c_v} \snorm{A_++A_-}^2 + \curlO(\varepsilon(1+\snorm{\gamma_1} + \snorm{\gamma_2})\snorm{A_+ + A_-}^2),
\end{align*}
that is, $B$ is determined by $\snorm{A_+ + A_-}^2$ even after including higher order terms.
We focus on the heteroclinic connections between the invading state $(A^\ast,0,0,0)$ and the origin, which can be observed numerically if $c_0$ is sufficiently large.
To make the argument more precise, we make the following assumptions about the spectra and the existence of heteroclinic orbits, which is in line with the numerical observations in the previous section.

\begin{enumerate}[label=(A\arabic*)]
	\item There exists a $c_0^\ast > 0$ such that for all $c_0 > c_0^\ast$ the origin is a (spectrally) stable, hyperbolic fixed point with respect to the dynamics given by \eqref{eq:redEqS2_Complex_1}--\eqref{eq:redEqS2_Complex_4} and the linearization about the invading state $(A^\ast,0,0,0)$ has a one-dimensional unstable eigenspace.
		\label{assump1}
	\item There exists a $c_0^\ast > 0$ such that for all $c_0 > c_0^\ast$ the system \eqref{eq:redEqS2_Complex_1}--\eqref{eq:redEqS2_Complex_4} has a heteroclinic orbit $(A_{r,\text{het}}, A_{i,\text{het}}, \tilde{A}_{r,\text{het}}, \tilde{A}_{i,\text{het}})$ such that
		\begin{align*}
			&\lim_{\hat{\xi} \rightarrow -\infty} (A_{r,\text{het}}, A_{i,\text{het}}, \tilde{A}_{r,\text{het}}, \tilde{A}_{i,\text{het}})(\hat{\xi}) = (A^\ast, 0,0,0), \\
			&\lim_{\hat{\xi} \rightarrow +\infty} (A_{r,\text{het}}, A_{i,\text{het}}, \tilde{A}_{r,\text{het}}, \tilde{A}_{i,\text{het}})(\hat{\xi}) = (0, 0,0,0).
		\end{align*}
		\label{assump2}
\end{enumerate}

Using the translational invariance of \eqref{eq:SHDis}--\eqref{eq:ConDis} we obtain the persistence of the fixed point $(A^\ast,0,0,0)$ for a specific choice of $\omega_0^\ast$ with respect to higher order perturbations as in Lemma \ref{lem:travelingWaves}.
Then, fix $c_0^\ast$ such that \ref{assump1}--\ref{assump2} hold.
Hence, the origin is a stable, hyperbolic fixed point and thus, the corresponding stable manifold is 4-dimensional in $\R^4$.
Therefore, using \ref{assump2}, the one-dimensional unstable manifold of the invading state $(A^\ast,0,0,0)$ and the stable manifold of the origin have to intersect transversally since the heteroclinic orbit is part of the unstable manifold and connects to the origin.
Thus, the intersection persists under small perturbations in $(\varepsilon, \gamma_1,\gamma_2)$ and we obtain the existence of modulating traveling fronts by reversing the center manifold reduction Lemma \ref{lem:centerManifold}.

\begin{theorem}\label{thm:S2}
	Let $c_0^\ast > 0$ such that \ref{assump1}--\ref{assump2} are satisfied.
	Furthermore, assume that $-c_v \neq 3c_u$ and $c = 3c_u + \varepsilon c_0$ with $c_0 > c_0^\ast$.
	Then, there exist $\varepsilon_0 > 0$ and $\gamma^\ast > 0$ such that for all $\varepsilon \in (0,\varepsilon_0)$ and $\gamma_1,\gamma_2 \in (-\gamma^\ast,\gamma^\ast)$ the system \eqref{eq:SHDis}--\eqref{eq:ConDis} has a family of modulating traveling front solutions $(u_\text{front},v_\text{front}) = (u_\text{front},v_\text{front})(\xi,p)$ satisfying \eqref{eq:modfrontBC} and are of the form
	\begin{align*}
		u_\text{front}(\xi,p) &= \varepsilon 2 A_{r,\text{het}}(\varepsilon\xi) \cos(p) + \curlO(\varepsilon^2), \\
		v_\text{front}(\xi,p) &= \varepsilon^2 \left(\dfrac{2\gamma_1}{c + c_v} A_{r,\text{het}}(\varepsilon\xi) + 2\Re\left(\dfrac{-2\gamma_1 + i \gamma_2}{2 - i(c_u+c_v)} A_{r,\text{het}}(\varepsilon\xi)^2\right)\cos(2p)\right) + \curlO(\varepsilon^3).
	\end{align*}
	Here, the heteroclinic orbit $A_{r,\text{het}}$ is given in \ref{assump2}, $\xi = x-ct$ and $p = x-\cphase t$ with $\cphase = c_u + \varepsilon^2 \omega_0^\ast$ from Lemma \ref{lem:travelingWaves}.
\end{theorem}

\begin{remark}\label{rem:reversedModFrontS2}
	Recalling that the system \eqref{eq:redEqS2_1}--\eqref{eq:redEqS2_2} is reversible with respect to $R : (A, \tilde{A},c_0) \mapsto (A, -\tilde{A},-c_0)$ we therefore obtain the existence of reversed modulating fronts for $c_0 < -c_0^\ast$, that is $(u_\text{front},v_\text{front})$ satisfying \eqref{eq:revModfrontBC}.
	This is in line with the findings in Theorem \ref{thm:S1}, since $c_0 < 0$ implies that $c < 3c_u$.
\end{remark}

\subsection{Scenario \ref{scenario3}}\label{sec:S3}

In the previous scenarios \ref{scenario1} and \ref{scenario2} the spreading speed of the modulating front was far away from the linear group velocity of the conservation law \eqref{eq:ConDis}, that is, $c\vert_{\varepsilon = 0} \neq -c_v$.
For this choice of spreading speed, the central mode associated to the conservation law is determined by the non-conserved parts on the center manifold.
The main reason for this is the conservation law structure, which allows the removal of the one zero center eigenvalue of the conservation law by integration.
Recalling the spectral result Lemma \ref{lem:specLcon}, the conservation law part gains an additional non-zero central eigenvalue if the spreading speed is chosen close (with respect to $\varepsilon$) to $-c_v$.
Therefore, in this case the conservation law plays a more active role in the dynamics on the center manifold.

We assume that $-c_v \neq 3c_u$ and thus the linear group velocities of the equations \eqref{eq:SHDis} and \eqref{eq:ConDis} are different.
Then, we choose $c = -c_v + \varepsilon^2 c_0$ with $c_0 \neq 0$ and $\varepsilon > 0$.
Since $c\vert_{\varepsilon = 0} = -c_v \neq 3c_u$, the central eigenvalues originating from $L^\text{SH}$ have been calculated in Section \ref{sec:S1} and are given by
\begin{align*}
	\lambda_{c,1}^\text{SH} = \varepsilon^2 \dfrac{\alpha_0 + i \omega_0^\ast}{3c_u - c} + \curlO(\varepsilon^3),
\end{align*}
and $\lambda_{c,-1}^\text{SH} = \overline{\lambda_{c,1}^\text{SH}}$.
Furthermore, using Lemma \ref{lem:specLcon} the central spectrum of $L^\text{con}$ is given by a zero eigenvalue and one non-zero eigenvalue $\lambda_{c,1}^\text{con} = -c-c_v = -\varepsilon^2 c_0$.

\begin{remark}\label{rem:spatialRescalingS3}
	We point out that the choice of $c$ such that $c + c_v$ vanishes quadratically as $\varepsilon \rightarrow 0$ guarantees that both non-zero central eigenvalues have the same $\varepsilon$-scaling.
	Therefore, as in \eqref{eq:uCoordCMS1}--\eqref{eq:vCoordCMS1} and \eqref{eq:uCoordCMS2}--\eqref{eq:vCoordCMS2} we can introduce a rescaling of the spatial coordinate $\xi$ on the center manifold to balance the powers of $\varepsilon$ on the linear level of the reduced equation, see also Remark \ref{rem:spatialRescaling}.
\end{remark}

\subsubsection{Derivation of the reduced system}

As in Section \ref{sec:S1}, let $\phi_c^\text{SH}$ be the eigenvector corresponding to the central eigenvalue $\lambda_{c,1}^\text{SH}$ of $L_1^\text{SH}$.
Since the central eigenvalues of $L^\text{con}$ are precisely the eigenvalues of $L_0^\text{con}$, we introduce the following coordinates
\begin{align*}
	U_c(\xi) &= \varepsilon A(\varepsilon^2 \xi) \phi_c^\text{SH} e^{ip} + c.c., \\
	V_c(\xi) &= \varepsilon^2 B(\varepsilon^2 \xi),
\end{align*}
with $A(\tilde{\xi}) \in \C$, $B(\tilde{\xi}) = (B_1,B_2)^T(\tilde{\xi}) \in \R^2$ and $\tilde{\xi} = \varepsilon^2 \xi$.
Furthermore, the reduction map $h = (h_u,h_v)$ from Lemma \ref{lem:centerManifold} can be approximated by
\begin{align*}
	h_u(U_c,V_c) &= \left(\varepsilon^2 \dfrac{i A^2}{9+6ic_u} e^{2ip} + c.c.,0,0,0\right)^T + \curlO(\varepsilon^3), \\
	h_v(U_c,V_c) &= \left(\varepsilon^2 A^2 \dfrac{-2\gamma_1 + i \gamma_2}{2-i(c_u+c_v)} e^{2ip} + c.c.,0\right)^T + \curlO(\varepsilon^3),
\end{align*}
following the calculations in Section \ref{sec:S1}.
Hence, the reduced system on the center manifold is given by
\begin{align*}
	\varepsilon^3 \partial_{\tilde{\xi}} A &= \varepsilon\lambda_{c,1}^\text{SH} A + \curlP_{\phi_{c}^\text{SH}} \curlN^\text{SH}(U,V), \\
	\varepsilon^4 \partial_{\tilde{\xi}} B &= \varepsilon^2 L_0^\text{con} B + \curlN_0^\text{con}(U,V),
\end{align*}
where $(U,V) = (U_c,V_c) + h(U_c,V_c)$ and $\curlP_{\phi_{c}^\text{SH}}$ is the spectral projection onto the eigenspace spanned by $\phi_{c}^\text{SH}$, see \eqref{eq:spectralProjections}.

Since the equation for $A$ is already been derived in Section \ref{sec:S1}, we focus on the equation for $B$.
Therefore, we recall that 
\begin{align*}
	L_0^\text{con} = \begin{pmatrix}
		0 & 1 \\
		0 & -\varepsilon^2 c_0
	\end{pmatrix},
\end{align*}
using $c = -c_v + \varepsilon^2 c_0$, and
\begin{align*}
	\curlN_0^\text{con} &= \left(0, -2\gamma_1 \sum_{k+j=0} (U_{k0}U_{j2} + U_{k1}U_{j1}) - 2\gamma_2 \sum_{k+j=0} U_{k0}U_{j1}\right)^T \\
	&= \left(0, -2 \gamma_1 \varepsilon^2 \partial_{\tilde{\xi}}\sum_{k+j=0} U_{k0}U_{j1} - \gamma_2 \varepsilon^2 \partial_{\tilde{\xi}}\sum_{k+j=0} U_{k0}U_{j0}\right)^T,
\end{align*}
see \eqref{eq:nonlinCon}, where we used in particular that $\partial_\xi = \varepsilon^2 \partial_{\tilde{\xi}}$.
Hence, $B$ satisfies
\begin{align*}
	\varepsilon^2 \partial_{\tilde{\xi}} B_1 &= B_2, \\
	\varepsilon^2 \partial_{\tilde{\xi}} B_2 &= -\varepsilon^2 c_0 B_2 + \left(-2 \gamma_1 \partial_{\tilde{\xi}}\sum_{k+j=0} U_{k0}U_{j1} - \gamma_2 \partial_{\tilde{\xi}}\sum_{k+j=0} U_{k0}U_{j0}\right).
\end{align*}
Writing this system as a second order equation and integrating once with respect to $\tilde{\xi}$, we find that
\begin{align}
	\partial_{\tilde{\xi}} B_1 = -c_0 B_1 + \varepsilon^{-4} \left(-2\gamma_1 \sum_{k+j = 0} U_{k0}U_{j1} - \gamma_2 \sum_{k+j = 0} U_{k0}U_{j0}\right).
	\label{eq:redEqS3B_intermediate}
\end{align}
Following \cite{eckmannWayne91}, it holds that $\phi_c^\text{SH} = (1,\lambda_{c,1}^\text{SH},(\lambda_{c,1}^\text{SH})^2,(\lambda_{c,1}^\text{SH})^3)^T$ and thus, we have that
\begin{align*}
	U_{10} &= \varepsilon A + \curlO(\varepsilon^3), \quad U_{11} = \varepsilon \lambda_{c,1}^\text{SH} A + \curlO(\varepsilon^5), \\
	U_{j0} &= \curlO(\varepsilon^2), \quad U_{j1} = \curlO(\varepsilon^4),
\end{align*}
with $j \neq \pm 1$ and $U_{-1} = \overline{U_{1}}$.
In particular, a non-zero $\gamma_2$-term in \eqref{eq:redEqS3B_intermediate} leads to contributions of order $\varepsilon^{-2}$ in the chosen scaling.
To avoid this, we set $\gamma_2 = 0$ in what follows.
Then, after neglecting higher order terms, we obtain the reduced system on the center manifold
\begin{subequations}
	\begin{align}
		\partial_{\tilde{\xi}} A &= \dfrac{\alpha_0 + i \omega_0^\ast}{3c_u - c} A + \dfrac{1}{3c_u - c} \left(AB_1 + \left(-3 - \dfrac{1}{9+6ic_u} + \dfrac{-2\gamma_1}{2-i(c_u+c_v)}\right) A\snorm{A}^2\right), \label{eq:redEqS3_1}\\
		\partial_{\tilde{\xi}} B_1 &= -c_0 B_1 - 4\gamma_1 \dfrac{\alpha_0}{3c_u - c} \snorm{A}^2, \label{eq:redEqS3_2}
	\end{align}
\end{subequations}
where we used that $\sum_{k+j=0} U_{k0}U_{j1} = 2\varepsilon^2\Re(\lambda_{c,1}^\text{SH}) \snorm{A}^2 + \curlO(\varepsilon^6)$.

\subsubsection{Existence of heteroclinic orbtis, their persistence and existence of modulating traveling fronts}

We now establish the existence of heteroclinic orbits in \eqref{eq:redEqS3_1}--\eqref{eq:redEqS3_2}, which connect $(A^\ast,0)$ for $\tilde{\xi} \rightarrow -\infty$ to the origin for $\tilde{\xi} \rightarrow +\infty$, if $(\gamma_1,\varepsilon)$ are close to zero.
Note that for $\gamma_1 = 0$, the set $\{A \in \C, B_1 \in \R \,:\, B_1 = 0\}$ is invariant.
Thus, in this case the problem reduces to finding heteroclinic orbits in \eqref{eq:redEqS3_1} for $\gamma_1 = B_1 = 0$, which has already been done in Section \ref{sec:S1}.
Hence, it remains to show that the orbit persists for $\gamma_1$ close to zero and after adding higher order terms in $\varepsilon > 0$.

As in Section \ref{sec:S1}, we use polar coordinates for $A \in \C$ and write $A = r_A e^{i\phi_A}$.
Then, the system \eqref{eq:redEqS3_1}--\eqref{eq:redEqS3_2} reads as
\begin{subequations}
	\begin{align}
		\partial_{\tilde{\xi}} r_A &= \dfrac{\alpha_0}{3c_u - c} r_A + \dfrac{1}{3c_u - c} \left(r_A B_1 + \left(-3-\dfrac{1}{9+4c_u^2}\right) r_A^3\right), \label{eq:redEqS3_polarCoord_1}\\
		\partial_{\tilde{\xi}} \phi_A &= \dfrac{\omega_0^\ast}{3c_u - c} + \dfrac{r_A^2}{3c_u - c} \dfrac{2c_u}{27+12c_u^2}, \label{eq:redEqS3_polarCoord_2}\\
		\partial_{\tilde{\xi}} B_1 &= -c_0, \label{eq:redEqS3_polarCoord_3}
	\end{align}
\end{subequations}
where we still set $\gamma_1 = 0$.
Since \eqref{eq:redEqS3_polarCoord_1} and \eqref{eq:redEqS3_polarCoord_3} are independent of $\phi_A$, we can analyze them separately.
Furthermore, note that the system is invariant with respect to $(\tilde{\xi},c_0,3c_u-c) \mapsto (-\tilde{\xi},-c_0,-(3c_u-c))$.
Hence, we restrict the analysis to the case $(3c_u - c) < 0$ or, equivalently, $-c_v > 3c_u$ and $\varepsilon > 0$ sufficiently small.
Linearizing the system \eqref{eq:redEqS3_polarCoord_1} and \eqref{eq:redEqS3_polarCoord_3} about the origin, we find
\begin{align*}
	L_0 = \begin{pmatrix}
		\frac{\alpha_0}{3c_u -c} & 0 \\
		0 & -c_0
	\end{pmatrix}.
\end{align*}
Thus, if $c_0 > 0$, the origin is a stable fixed point, while for $c_0 < 0$ it is a saddle point.
Next, we linearize about the invading state $(A^\ast,0)$, which yields
\begin{align*}
	L_\text{invading} = \begin{pmatrix}
		-\frac{4\alpha_0}{3c_u - c} & \frac{A^\ast}{3c_u - c} \\
		0 & -c_0
	\end{pmatrix}.
\end{align*}
Hence, if $c_0 > 0$, the invading state is a saddle point with a one-dimensional unstable eigenspace, while for $c_0 < 0$, it is an unstable fixed point.

If $c_0 > 0$, the setting is comparable to the ones in \cite{hilder20} and Section \ref{sec:S2}, that is, the linearization about the invading state has a one-dimensional unstable manifold and the origin is a stable, hyperbolic fixed point.
Therefore, we establish the persistence similar to these cases.
Since there exists a heteroclinic connection from $(A^\ast,0)$ to the origin if $B_1 = 0$, the unstable manifold of the linearization about the invading state and the stable manifold of the linearization about the origin have to intersect for $B_1 = 0$.
Furthermore, the intersection is also transversal since the stable manifold of the linearization about the origin is a two-dimensional in $\R^2$ (recall that we neglect the phase-equation \eqref{eq:redEqS3_polarCoord_2} for now).
Hence, the intersection persists under small perturbations, that is for $\gamma_1$ close to zero and $\varepsilon > 0$ small and we obtain a solution $(A_\text{het},B_{1,\text{het}})(\tilde{\xi})$ satisfying
\begin{align}
	\lim_{\tilde{\xi} \rightarrow -\infty} (A_\text{het},B_{1,\text{het}})(\tilde{\xi}) = (A^\ast,0) \text{ and } \lim_{\tilde{\xi} \rightarrow +\infty} (A_\text{het},B_{1,\text{het}})(\tilde{\xi}) = (0,0)
	\label{eq:asymptoticConditions_S3}
\end{align}
after including \eqref{eq:redEqS3_polarCoord_2}.

In the case $c_0 < 0$, consider the time-reversed system, i.e.~$\tilde{\xi} \mapsto \tilde{\xi}^R := -\tilde{\xi}$.
Then, the origin has a one-dimensional unstable manifold and the invading state is a stable fixed point.
Therefore, we can proceed as above and find a persistent heteroclinic orbit, which converges to the origin for $\tilde{\xi}^R \rightarrow -\infty$ and to $(A^\ast,0)$ for $\tilde{\xi}^R \rightarrow +\infty$.
Reversing the coordinate change then gives a persistent heteroclinic orbit satisfying \eqref{eq:asymptoticConditions_S3}.
Combining both cases and using Lemma \ref{lem:centerManifold}, we obtain the existence of modulating traveling fronts.

\begin{theorem}\label{thm:S3}
	Let $c = -c_v + \varepsilon^2 c_0$ with $c_0 \neq 0$.
	Furthermore, assume that $-c_v > 3c_u$, $c_u \neq -c_v$ and $\gamma_2 = 0$.
	Then, there exist $\varepsilon_0 > 0$ and $\gamma_1^\ast > 0$ such that for all $\varepsilon > 0$ and $\gamma_1 \in (-\gamma_1^\ast,\gamma_1^\ast)$ the system \eqref{eq:SHDis}--\eqref{eq:ConDis} has a family of modulating traveling front solutions $(u_\text{front},v_\text{front})(\xi,p)$ satisfying \eqref{eq:modfrontBC} and 
	\begin{align*}
		u_\text{front}(\xi,p) &= \varepsilon 2 \Re(\Ahet(\varepsilon^2 \xi)) \cos(p) + \curlO(\varepsilon^2), \\
		v_\text{front}(\xi,p) &= \varepsilon^2 \left(B_{1,\text{het}}(\varepsilon^2 \xi) + 2\Re\left(\dfrac{-2\gamma_1 + i \gamma_2}{2 - i (c_u+c_v)} \Ahet(\varepsilon^2\xi)^2\right)\cos(2p)\right) + \curlO(\varepsilon^3),
	\end{align*}
	with $(\Ahet,B_{1,\text{het}})$ satisfying \eqref{eq:asymptoticConditions_S3}, $\xi = x-ct$ and $p = x-\cphase t$ with $\cphase = c_u + \varepsilon^2 \omega_0^\ast$ given in Lemma \ref{lem:travelingWaves}.
\end{theorem}

\begin{remark}\label{rem:symmetryS3}
	Due to the symmetry $(\tilde{\xi},c_0,3c_u-c) \mapsto (-\tilde{\xi},-c_0,-(3c_u-c))$ in the reduced system \eqref{eq:redEqS3_polarCoord_1}--\eqref{eq:redEqS3_polarCoord_3}, we obtain a reversed modulating front in the case $-c_v < 3c_u$, see also Remark \ref{rem:reversedModFrontS2}.
\end{remark}

\subsection{Scenario \ref{scenario4}} \label{sec:S4}
In this scenario, we again assume that $-c_v \neq 3c_u$ and that the spreading speed of the modulating traveling front is close to $-c_v$ when $\varepsilon > 0$ is small, as in Scenario \ref{scenario3}.
Using Lemma \ref{lem:specLSH}, the set of central eigenvalues of $L^\text{SH}$, the linear part originating from the dispersive Swift-Hohenberg equation \eqref{eq:SHDis} is given by two complex conjugated eigenvalues $\lambda_{c,\pm 1}^\text{SH}$ corresponding to $L_{\pm 1}^\text{SH}$, respectively.
In particular, it holds that $\lambda_{c,\pm 1}^\text{SH} = \curlO(\varepsilon^2)$, see Sections \ref{sec:S1} and \ref{sec:S3}.
Furthermore, using Lemma \ref{lem:specLcon}, the central spectrum of $L^\text{con}$ is given by the two eigenvalues of $L_0^\text{con}$ given by $\lambda_{c,0}^\text{con} = 0$ and $\lambda_{c,1}^\text{con} = -c-c_v$.

Recall that in Scenario \ref{scenario3}, we chose $c$ such that $-c-c_v$ vanishes quadratically as $\varepsilon \rightarrow 0$, which ensures that all central eigenvalues are of the same order $\curlO(\varepsilon^2)$.
This allows for a rescaling of the spatial variable $\xi$ such that all linear parts of the reduced system on the center manifold have the same scaling in $\varepsilon$, see Remarks \ref{rem:spatialRescaling} and \ref{rem:spatialRescalingS3}.
In this section, we instead set $c = -c_v + \varepsilon c_0$ with $c_0 \neq 0$, which leads to $\lambda_{c,1}^\text{con} = -\varepsilon c_0$, which vanishes linearly as $\varepsilon \rightarrow 0$ instead of quadratically, and hence, the central eigenvalues have a different scaling in $\varepsilon$.

\subsubsection{The reduced system on the center manifold}

As in Section \ref{sec:S3}, to capture the relevant dynamics on the center manifold in this setting, we introduce
\begin{align*}
	U(\xi) &= \varepsilon A(\varepsilon^2 \xi) \phi_c^\text{SH} e^{ip} + c.c. + h_u(A,B), \\
	V(\xi) &= \varepsilon^2 B(\varepsilon^2 \xi) + h_v(A,B),
\end{align*}
where $A(\tilde{\xi}) \in \C$, $B(\tilde{\xi}) = (B_1,B_2)^T \in \R^2$ and $\tilde{\xi} = \varepsilon^2 \xi$.
Proceeding as in Section \ref{sec:S3}, the dynamics on the center manifold is, up to higher order terms, given by
\begin{align*}
	\partial_{\tilde{\xi}} A &= \dfrac{\alpha_0 + i \omega_0^\ast}{3c_u - c} A + \dfrac{1}{3c_u - c} \left(AB_1 + \left(-3 - \dfrac{1}{9+6ic_u} + \dfrac{-2\gamma_1 + i\gamma_2}{2 - i(c_u+c_v)}\right) A \snorm{A}^2\right), \\
	\partial_{\tilde{\xi}} B_1 &= -\varepsilon^{-1} c_0 B_1 + \varepsilon^{-4} \left(-2\gamma_1 \sum_{k+j=0} U_{k0}U_{j1} - \gamma_2 \sum_{k+j=0} U_{k0}U_{j0}\right).
\end{align*}
Similar to Section \ref{sec:S3}, the $\gamma_1$-nonlinearity is of order $\varepsilon^4$ since $U_{10} \sim \varepsilon$ and $U_{11} \sim \varepsilon^3$ are the lowest order $U$-terms.
This balances with the $\varepsilon^{-4}$ in front of the nonlinearity and leads to an $\curlO(1)$-nonlinearity, see also \eqref{eq:redEqS3_2}.
In contrast, the $\gamma_2$-nonlinearity is only of order $\varepsilon^2$, since the corresponding term in \eqref{eq:ConDis} contains only one derivative instead of two.
In Section \ref{sec:S3}, we therefore set $\gamma_2 = 0$ to remove the singularity at $\varepsilon = 0$, in the current setting however, we choose $\gamma_2 = \varepsilon\gamma_2^0$ with $\gamma_2^0 \in \R$, which then balances with the $\curlO(\varepsilon^{-1})$ linearity.
Thus, the dynamics on the center manifold, up to higher order terms, is given by
\begin{subequations}
	\begin{align}
		\partial_{\tilde{\xi}} A &= \dfrac{\alpha_0 + i \omega_0^\ast}{3c_u - c} A + \dfrac{1}{3c_u - c} \left(AB_1 + \left(-3 - \dfrac{1}{9+6ic_u} + \dfrac{-2\gamma_1 + i\gamma_2}{2 - i(c_u+c_v)}\right) A \snorm{A}^2\right), \label{eq:redEqS4_1}\\
		\partial_{\tilde{\xi}} B_1 &= -\varepsilon^{-1} (c_0 B_1 + 2\gamma_2^0 \snorm{A}^2) - 4\gamma_1\dfrac{\alpha_0}{3c_u - c}\snorm{A}^2. \label{eq:redEqS4_2}
	\end{align}
\end{subequations}

\subsubsection{Heteroclinic orbits, their persistence and existence of modulating traveling fronts}

We now study the dynamics of \eqref{eq:redEqS4_1}--\eqref{eq:redEqS4_2} for $\varepsilon > 0$ small and in particular, the existence of heteroclinic orbits.
Formally, the $\varepsilon^{-1}$-terms in \eqref{eq:redEqS4_2} are the most relevant and equating this to zero then yields an algebraic equation for $B_1$, that is,
\begin{align*}
	B_1 = -\dfrac{2\gamma_2^0}{c_0} \snorm{A}^2.
\end{align*}
We point out that a similar relation has been obtained in Sections \ref{sec:S1} and \ref{sec:S2}, where $B$ was determined by $A$, if one accounts for $c + c_v = \varepsilon c_0$ and $\gamma_2 = \varepsilon \gamma_2^0$.
Therefore, we might expect that, to lowest order, the dynamics of $B_1$ is determined by $A$, although $B_1$ is likely not fully governed by $A$ as in Sections \ref{sec:S1} and \ref{sec:S2}.

For a more rigorous treatment, we write $A$ in polar coordinates, i.e.~$A = r_A e^{i\phi_A}$ and multiply \eqref{eq:redEqS4_2} by $\varepsilon$, which yields
\begin{subequations}
	\begin{align}
		\partial_{\tilde{\xi}} r_A &= \dfrac{\alpha_0}{3c_u - c} r_A + \dfrac{1}{3c_u - c} \left(r_A B_1 + \left(-3-\dfrac{1}{9+4c_u^2} - \dfrac{4\gamma_1}{4 + (c_u+c_v)}\right) r_A^3\right), \label{eq:fastSlowSys1}\\
		\partial_{\tilde{\xi}} \phi_A &= \dfrac{\omega_0^\ast}{3c_u - c} - \dfrac{r_A^2}{3c_u - c} \left(\dfrac{2c_u}{27+12c_u^2} + \dfrac{2\gamma_1(c_u+c_v)}{4 + (c_u+c_v)^2}\right), \label{eq:fastSlowSys2} \\
		\varepsilon\partial_{\tilde{\xi}} B_1 &= -c_0 B_1 - 2\gamma_2^0 r_A^2 - \varepsilon 4\gamma_1 \dfrac{\alpha_0}{3c_u-c} r_A^2. \label{eq:fastSlowSys3}
	\end{align}
\end{subequations}
To analyze this system, we exploit its fast-slow structure and use geometric singular perturbation theory, see e.g.~\cite{kuehn15}, to establish the existence and persistence of heteroclinic connections from $(A^\ast,0,B_1^\ast)$ to the origin.
Here $A^\ast$ is given in Lemma \ref{lem:travelingWaves} and $B_1^\ast = -\frac{2\gamma_2^0}{c_0} \snorm{A^\ast}^2 + \curlO(\varepsilon)$.
We highlight that, similar to Remark \ref{rem:nonlinDispersionMotivationS1}, the inclusion of a nonlinear dispersion term (i.e.~$\gamma_2 \neq 0$) enriches the mathematical analysis.

As in Section \ref{sec:S3}, we first note that the $r_A$- and $B_1$-equations are independent of the phase $\phi_A$ and thus, it is sufficent to construct a heteroclinic orbit for the $(r_A,B_1)$-system determined by \eqref{eq:fastSlowSys1} and \eqref{eq:fastSlowSys3}.
Then, to construct a persistent heteroclinic orbit, we follow \cite[Section 6.1]{kuehn15}.
First, setting $\varepsilon = 0$ we obtain the slow subsystem, which reads as
\begin{subequations}
	\begin{align}
		\partial_{\tilde{\xi}} r_A &= \dfrac{\alpha_0}{3c_u - c} r_A + \dfrac{1}{3c_u - c} \left(r_A B_1 + \left(-3-\dfrac{1}{9+4c_u^2} - \dfrac{4\gamma_1}{4 + (c_u+c_v)}\right) r_A^3\right), \label{eq:slowSubsys1} \\
		0 &= -c_0 B_1 - 2\gamma_2^0 r_A^2. \label{eq:slowSubsys2}
	\end{align}
\end{subequations}
Due to \eqref{eq:slowSubsys2}, the slow subsystem is restricted to the critical manifold $C_0$ defined by
\begin{align*}
	C_0 := \left\{(r_A,B_1) \in \R^2 \,:\, B_1 = -\dfrac{2\gamma_2^0}{c_0} r_A^2\right\}.
\end{align*}
Next, switching to the fast time-scale $\tilde{\xi}_f := \tilde{\xi}/\varepsilon$ in \eqref{eq:fastSlowSys1} and \eqref{eq:fastSlowSys3} and again setting $\varepsilon = 0$ gives the fast subsystem
\begin{subequations}
	\begin{align}
		\partial_{\tilde{\xi}_f} r_A &= 0, \label{eq:fastSubsys1} \\
		\partial_{\tilde{\xi}_f} B_1 &= -c_0 B_1 - 2\gamma_2^0 r_A^2. \label{eq:fastSubsys2}
	\end{align}
\end{subequations}
Note that the points on the critical manifold $C_0$ are equilibria of the fast subsystem \eqref{eq:fastSubsys1}--\eqref{eq:fastSubsys2}.
Additionally, if $c_0 > 0$, the critical manifold is globally attractive with respect to the flow of the fast subsystem \eqref{eq:fastSubsys1}--\eqref{eq:fastSubsys2}, that is, for all initial data the solution of \eqref{eq:fastSubsys1}--\eqref{eq:fastSubsys2} tends to a point on $C_0$ for $\tilde{\xi}_f \rightarrow +\infty$, see Figure \ref{fig:phasePlane}.
Similarly, if $c_0 < 0$, the critical manifold is globally repelling with respect to the flow of the fast subsystem, i.e.~for all initial data the solution of \eqref{eq:fastSubsys1}--\eqref{eq:fastSubsys2} tends to a point on $C_0$ for $\tilde{\xi}_f \rightarrow -\infty$.

In the subsequent analysis, we restrict to the case $c_0 > 0$.
The case $c_0 < 0$ can be handled by reversing time as in Section \ref{sec:S3}, which is discussed in Remark \ref{rem:c0vNegative}.
Then, we introduce the unstable manifold of the invading state $p_f = (A^\ast, -2\gamma_2^0(c_0)^{-1} (A^\ast)^2)$ and the stable manifold of the origin $p_0 = (0,0)$ with respect to the slow flow as
\begin{align*}
	W^u(p_f) &:= \{p \in C_0 \,:\, \Phi_{\tilde{\xi}}^\text{slow}(p) \rightarrow p_f \text{ for } \tilde{\xi} \rightarrow -\infty\}, \\
	W^s(p_0) &:= \{p \in C_0 \,:\, \Phi_{\tilde{\xi}}^\text{slow}(p) \rightarrow p_0 \text{ for } \tilde{\xi} \rightarrow +\infty\},
\end{align*}
where $\Phi_{\tilde{\xi}}^\text{slow}$ denotes the flow of the slow subsystem \eqref{eq:slowSubsys1}--\eqref{eq:slowSubsys2}.
Finally, we define
\begin{align*}
	N_{p_f} &:= \bigcup_{p \in W^u(p_f)} \{q \in \R^2 \,:\, \Phi_{\tilde{\xi}_f}^\text{fast}(q) \rightarrow p \text{ for } \tilde{\xi}_f \rightarrow -\infty\}, \\
	N_{p_0} &:= \bigcup_{p \in W^s(p_f)} \{q \in \R^2 \,:\, \Phi_{\tilde{\xi}_f}^\text{fast}(q) \rightarrow p \text{ for } \tilde{\xi}_f \rightarrow +\infty\},
\end{align*}
where $\Phi_{\tilde{\xi}_f}^\text{fast}$ denotes the flow of the fast subsystem \eqref{eq:fastSubsys1}--\eqref{eq:fastSubsys2}.
Using this notation, we show the existence of a persistent heteroclinic orbit by applying \cite[Theorem 6.1.1]{kuehn15}, which we recapitulate (adapted to our setting) here.

\begin{figure}
	\centering
	\begin{tikzpicture}
		\begin{axis}[
				xmin = -1,
				xmax = 5,
				ymin = -5,
				ymax = 1,
				xtick = {4},
				xticklabels = {$A^\ast$},
				ytick = {},
				yticklabels = \empty,
				axis x line = center,
				x label style={anchor=south west},
				axis y line = center,
				xlabel = $r_A$,
				ylabel = $B_1$,
			]
			
			\addplot[mark=none, domain=0:5, color=red, postaction={decorate},
				decoration={markings, 
					mark=at position 0.25 with {\arrowreversed{Stealth[scale=1.5]}},
					mark=at position 0.5 with {\arrowreversed{Stealth[scale=1.5]}},
					mark=at position 0.8 with {\arrow{Stealth[scale=1.5]}}
				}
			]{-0.25*x^2};
			
			\pgfplotsinvokeforeach {0.75,1.5,2.25,3,3.75,4.5}
			{
				\draw[color=blue, postaction={decorate},
					decoration={markings, 
						mark=at position 0.25 with {\arrow{>>}},
						mark=at position 0.75 with {\arrow{>>}}
					}] (axis cs:#1,1) -- (axis cs:#1,-0.25*#1^2);
			}
			\pgfplotsinvokeforeach {0.5,1.25,2,2.75,3.5}
			{
				\draw[color=blue, postaction={decorate},
					decoration={markings, 
						mark=at position 0.25 with {\arrow{>>}},
						mark=at position 0.75 with {\arrow{>>}}
					}] (axis cs:#1,-5) -- (axis cs:#1,-0.25*#1^2);
			}
			
			\addplot[mark=x,only marks, mark size = 3pt, mark options={line width = 1pt}] coordinates {(0,0) (4,-4)};
			
			\node[anchor=south west] at (axis cs:2.4,-1.5625){\textcolor{red}{$C_0$}};
			\node[anchor=north east] at (axis cs:4,-4){$p_f$};
			\node[anchor=north east] at (axis cs:0,0){$p_0$};
		\end{axis}
	\end{tikzpicture}
	\caption{Depiction of the $(r_A,B_1)$-phase plane. The critical manifold $C_0$ and the flow of the slow subsystem \eqref{eq:slowSubsys1}--\eqref{eq:slowSubsys2} are plotted in red. Additionally, the flow of the fast subsystem \eqref{eq:fastSubsys1}--\eqref{eq:fastSubsys2} are drawn in blue. Here, the parameters $\gamma_2^0$ and $c_0$ are chosen such that $p_f$ is an unstable fixed point with respect to the dynamics of the slow subsystem, that is, $c_0 > 0$ and \eqref{eq:conditionGamma2S4} holds.}
	\label{fig:phasePlane}
\end{figure}
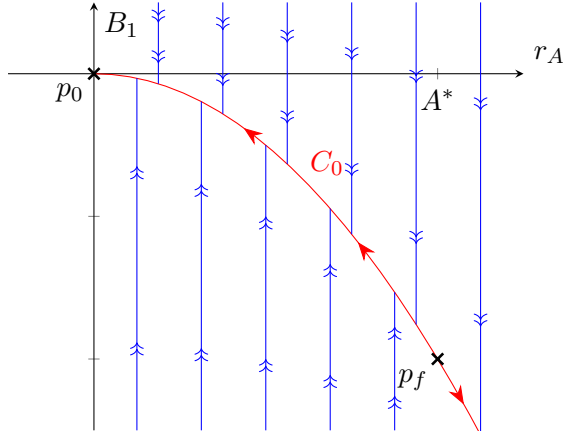

\begin{lemma}\label{lem:kuehnThm}
	Assuming that the fixed points $p_f$ and $p_0$ are hyperbolic in the slow flow and that the manifolds $N_{p_f}$ and $N_{p_0}$ intersect transversally.
	Then, provided that the slow subsystem \eqref{eq:slowSubsys1}--\eqref{eq:slowSubsys2} has a heteroclinic orbit from $p_f$ to $p_0$, there exists an $\varepsilon_0 > 0$ such that for all $\varepsilon \in (0,\varepsilon_0)$ the full system \eqref{eq:fastSlowSys1}--\eqref{eq:fastSlowSys3} has a heteroclinic orbit connecting $p_f$ for $\tilde{\xi} \rightarrow -\infty$ to $p_0$ for $\tilde{\xi} \rightarrow +\infty$.
\end{lemma}

Similar to the Scenario \ref{scenario3} discussed in Section \ref{sec:S3}, we restrict to $-c_v > 3c_u$ in the subsequent analysis since the case $-c_v < 3c_u$ can be dealt with using symmetry arguments (see Remark \ref{rem:symmetryS3}).
Additionally, we assume that $\gamma_1$ is close to zero and that
\begin{align}
	\gamma_2^0 > -c_0 \left(1 + \dfrac{1}{3(9+4c_u^2)}\right).
	\label{eq:conditionGamma2S4}
\end{align}
The resulting phase plane for the slow and fast subsystems is qualitatively given in Figure \ref{fig:phasePlane}.
Linearizing the slow subsystem \eqref{eq:slowSubsys1}--\eqref{eq:slowSubsys2} on the critical manifold about $p_0$ and $p_f$, respectively leads to
\begin{align*}
	L_{p_0} &= \dfrac{\alpha_0}{3c_u - c}, \\
	L_{p_f} &= \dfrac{2\alpha_0}{3c_u - c} - \dfrac{6 (A^\ast)^2\gamma_2^0}{c_0(3c_u - c)},
\end{align*}
where we used the approximate representation of $A^\ast$ given in Lemma \ref{lem:travelingWaves}.
In particular, since $-c_v > 3c_u$ it holds that $L_{p_0} < 0$ and $L_{p_f} > 0$, where we used \eqref{eq:conditionGamma2S4} for the latter.
Therefore, we find that $p_0$ is a stable fixed point and that $p_f$ is an unstable fixed point with respect to the dynamics of the slow subsystem \eqref{eq:slowSubsys1}--\eqref{eq:slowSubsys2}.
Therefore, the (one-dimensional) slow subsystem has a heteroclinic orbit connecting $p_f$ to $p_0$.
Furthermore, the set $\{(r_A,B_1) \in C_0 \,:\, r_A \in (0,A^\ast)\}$ is contained in both $W^u(p_f)$ and $W^s(p_0)$.
Finally, since the critical manifold $C_0$ is globally attractive with respect to the fast flow it holds that $N_{p_f} = W^u(p_f)$ and that $(0,A^\ast) \times \R \subset N_{p_0}$.
Hence, $N_{p_f}$ and $N_{p_0}$ intersect transversally and Lemma \ref{lem:kuehnThm} is applicable.
Then, using that the case $c_0 < 0$ can be dealt with time-reversal (see Remark \ref{rem:c0vNegative}) and reverting the center manifold reduction Lemma \ref{lem:centerManifold} we obtain the following result.

\begin{theorem}\label{thm:S4}
	Let $c = -c_v + \varepsilon c_0$ with $c_0 \neq 0$.
	Furthermore, assume that $-c_v > 3c_u$, $c_u \neq -c_v$ and $\gamma_2 = \varepsilon \gamma_2^0$ with $\gamma_2^0$ satisfying \eqref{eq:conditionGamma2S4}.
	Then, there exist $\varepsilon_0 > 0$ and $\gamma_1^\ast > 0$ such that for all $\varepsilon \in (0,\varepsilon_0)$ and $\gamma_1 \in (-\gamma_1^\ast,\gamma_1^\ast)$ the system \eqref{eq:SHDis}--\eqref{eq:ConDis} has a family of modulating traveling front solutions $(u_\text{front},v_\text{front})(\xi,p)$ satisfying \eqref{eq:modfrontBC} and
	\begin{align*}
		u_\text{front}(\xi,p) &= \varepsilon 2 \Re(\Ahet(\varepsilon^2\xi)) \cos(p) + \curlO(\varepsilon^2), \\
		v_\text{front}(\xi,p) &= \varepsilon^2 \left( \Bhet(\varepsilon^2\xi) + 2\Re\left(\dfrac{-2\gamma_1 + i \gamma_2}{2 - i (c_u+c_v)} \Ahet(\varepsilon^2\xi)^2\right)\cos(2p)\right) + \curlO(\varepsilon^3),
	\end{align*}
	where $(\Ahet,\Bhet)$ is a heteroclinic orbit of \eqref{eq:redEqS4_1}--\eqref{eq:redEqS4_2} with
	\begin{align*}
		\lim_{\tilde{\xi} \rightarrow -\infty} (\Ahet,\Bhet)(\tilde{\xi}) = (A^\ast,-2\gamma_2^0(c_0)^{-1} (A^\ast)^2) \text{ and } \lim_{\tilde{\xi} \rightarrow +\infty} (\Ahet,\Bhet)(\tilde{\xi}) = (0,0).
	\end{align*}
	Here $\xi = x-ct$ and $p = x-\cphase t$ with $\cphase = c_u + \varepsilon^2 \omega_0^\ast$ given in Lemma \ref{lem:travelingWaves}.
\end{theorem}

\begin{remark}\label{rem:c0vNegative}
	Although we have restricted the calculations to the case $c_0 > 0$, the case that $c_0 < 0$ can be handled using a time-reversed system similar to Section \ref{sec:S3}.
	Specifically, we recall that the critical manifold $C_0$ is globally repelling with respect to the fast subflow for $c_0 < 0$.
	Thus, we substitute $\tilde{\xi} \mapsto -\tilde{\xi} =: \tilde{\xi}^R$ in the slow subsystem \eqref{eq:slowSubsys1}--\eqref{eq:slowSubsys2} and $\tilde{\xi}_f \mapsto -\tilde{\xi}_f := \tilde{\xi}_f^R$ in the fast subsystem \eqref{eq:fastSubsys1}--\eqref{eq:fastSubsys2}.
	Then, we proceed as for $c_0 > 0$ and apply Lemma \ref{lem:kuehnThm} to obtain a heteroclinic connection from $p_0$ for $\tilde{\xi}^R \rightarrow -\infty$ to $p_f$ for $\tilde{\xi}^R \rightarrow +\infty$.
	Note that we have to reverse the roles of $N_{p_f}$ and $N_{p_0}$ in the process and that the condition \eqref{eq:conditionGamma2S4} on $\gamma_2^0$ is replaced by
	\begin{align*}
		\gamma_2^0 > c_0 \left(1 + \dfrac{1}{3(9+4c_u^2)}\right).
	\end{align*}
	However, we point out that this does not change the set of admissible $\gamma_2^0$ since $c_0 < 0$.
	Then, reverting back to the original time flow, we obtain Theorem \ref{thm:S4} for $c_0 < 0$.
\end{remark}

\subsection{Scenario \ref{scenario5}}\label{sec:S5}

Finally, we discuss the case that the spreading speed of the modulating front is $\varepsilon$-close to both group velocities $-c_v$ and $3c_u$, that is, we assume that $-c_v = 3c_u$ and set $c = 3c_u + \varepsilon c_0$.
Using the calculations in Section \ref{sec:S2} the four central eigenvalues of the Swift-Hohenberg part $L^\text{SH}$ are given by
\begin{align*}
	\lambda_{c,\pm}^\text{SH} = \varepsilon \delta_\pm + \curlO(\varepsilon^2), \quad \delta_\pm = -\dfrac{c_0}{8 + 6ic_u} \pm \dfrac{\Delta}{8+6ic_u}
\end{align*}
originating from $L_1^\text{SH}$ and their complex conjugates corresponding to $L_{-1}^\text{SH}$.
Here, recall that $\Delta = \sqrt{c_0^2 - 4(3c_u + 4)(\alpha_0 + i \omega_0^\ast)}$.
Additionally, using Lemma \ref{lem:specLcon}, there are two central eigenvalues from the conservation law part $L^\text{con}$ given by
\begin{align*}
	\lambda_{c,0}^\text{con} = 0 \text{ and } \lambda_{c,1}^{\text{con}} = -\varepsilon c_0.
\end{align*}
Since the derivation of the reduced system on the center manifold is very similar to the derivations in Sections \ref{sec:S2} and \ref{sec:S3}, we refrain from presenting the derivations in all details.
Recall that in Section \ref{sec:S3}, we set $\gamma_2 = 0$, since $\gamma_2 \neq 0$ leads to unbounded coefficients in the chosen scaling as $\varepsilon \rightarrow 0$.
Therefore, we again set $\gamma_2 = 0$ in the subsequent analysis.
Then, we choose the coordinates on the center manifold als
\begin{align*}
	U_c(\xi) &= \varepsilon\left(A_+(\varepsilon\xi) \phi_{c,+}^\text{SH} + A_-(\varepsilon\xi) \phi_{c,-}^\text{SH}\right) e^{ip} + c.c., \\
	V_c(\xi) &= \varepsilon^2 B(\varepsilon\xi),
\end{align*}
where $A_\pm(\tilde{\xi}) \in \C$, $B(\tilde{\xi}) = (B_1,B_2)^T(\tilde{\xi}) \in \R$ and $\tilde{\xi} = \varepsilon \xi$.
Here, as in \eqref{eq:uCoordCMS2}, $\phi_{c,\pm}^\text{SH}$ are the eigenvectors of $L_1^\text{SH}$ corresponding to the central eigenvalues $\lambda_{c,\pm}^\text{SH}$.
The dynamics on the center manifold, up to higher order terms, is then given by
\begin{align*}
	\partial_{\tilde{\xi}} A_+ &= \delta_+ A_+ - \dfrac{(A_+ + A_-) B_1}{\Delta} - \dfrac{\hat{a}_\text{cub}}{\Delta} (A_+ + A_-) \snorm{A_+ + A_-}^2, \\
	\partial_{\tilde{\xi}} A_- &= \delta_- A_- + \dfrac{(A_+ + A_-) B_1}{\Delta} + \dfrac{\hat{a}_\text{cub}}{\Delta} (A_+ + A_-) \snorm{A_+ + A_-}^2, \\
	\partial_{\tilde{\xi}} B_1 &= -c_0 B_1 - 2\varepsilon^{-3} \gamma_1 \sum_{k+j = 0} U_{k0} U_{j1},
\end{align*}
where
\begin{align*}
	\hat{a}_\text{cub} := -3 - \dfrac{1}{9+6ic_u} - \dfrac{2\gamma_1}{2-i(c_u+c_v)}.
\end{align*}
Next, we substitute as in \eqref{eq:substitutionS2},
\begin{align*}
	A := A_+ + A_- \text{ and } \tilde{A} := -\dfrac{c_0}{8+6ic_u} (A_+ + A_-) + \dfrac{\Delta}{8+6ic_u} (A_+ - A_-),
\end{align*}
and use that the second component of $U_c = (U_{c,0},U_{c,1},U_{c,2},U_{c,3})$ can be approximated by
\begin{align*}
	U_{c,1} &= \varepsilon^2 (A_+ \delta_+ + A_-\delta_-) e^{ip} + c.c. + \curlO(\varepsilon^3), \\
	&= \varepsilon^2 \left(-\dfrac{c_0}{8+6ic_u} (A_+ + A_-) + \dfrac{\Delta}{8+6ic_u} (A_+ - A_-)\right) e^{ip} + c.c. + \curlO(\varepsilon^3), \\
	&= \varepsilon^2 \tilde{A} e^{ip} + c.c. + \curlO(\varepsilon^3),
\end{align*}
which yields that $\sum_{k+j = 0} U_{k0}U_{j1} = 2\varepsilon^3 \Re(A \overline{\tilde{A}}) + \curlO(\varepsilon^4)$.
This gives the reduced system
\begin{subequations}
	\begin{align}
		\partial_{\tilde{\xi}} A &= \tilde{A}, \label{eq:redEqS5_1}\\
		\partial_{\tilde{\xi}} \tilde{A} &= -\dfrac{2c_0}{8+6ic_u} \tilde{A} + \dfrac{\Delta^2 - c_0^2}{(8+6ic_u)^2} A - \dfrac{2}{8+6ic_u} A B_1 - \dfrac{2 \hat{a}_\text{cub}}{8+6ic_u} A \snorm{A}^2, \label{eq:redEqS5_2}\\
		\partial_{\tilde{\xi}} B_1 &= -c_0 B_1 - 4\gamma_1 \Re(A\overline{\tilde{A}}). \label{eq:redEqS5_3}
	\end{align}
\end{subequations}

\subsubsection{Heteroclinic orbits, their persistence and existence of modulating traveling fronts}

Similarly to Section \ref{sec:S2}, we first set $\gamma_1 = 0$ (and $\varepsilon = 0$ by neglecting higher order terms) for which $\{(A,\tilde{A},B_1) \in \C^2 \times \R \,:\, B_1 = 0\}$ is an invariant set with respect to the dynamics of the system \eqref{eq:redEqS5_1}--\eqref{eq:redEqS5_3}.
Restricting to this set, the evolution on the center manifold is then, up to higher order terms, given by \eqref{eq:redEqS2_1}--\eqref{eq:redEqS2_2}.
This system has been studied numerically in Section \ref{sec:S2} and exhibits heteroclinic connections from $(A^\ast,0)$ for $\tilde{\xi} \rightarrow -\infty$ to $(0,0)$ for $\tilde{\xi} \rightarrow +\infty$ if $c_0 > c_0^\ast$ for some $c_0^\ast$.

It remains to show that these heteroclinics also persist in the extended system \eqref{eq:redEqS5_1}--\eqref{eq:redEqS5_3}.
Hence, similar to Section \ref{sec:S2}, we split $A$ and $\tilde{A}$ into their real and imaginary, which results in a real-valued dynamical system in $\R^5$.
We also note that using symmetry arguments, we can restrict to the case $c_0 > 0$.
Then, we make the assumptions \ref{assump1}--\ref{assump2}, that is, we assume that if $c_0 > c^\ast$, heteroclinic orbits exist for $\gamma_1 = B_1 = 0$, the origin is a stable, hyperbolic fixed point and the invading state $(A^\ast,0,0,0,0)$ is a hyperbolic fixed point with a one-dimensional unstable eigenspace.
In particular, we highlight that the addition of \eqref{eq:redEqS5_3} only adds negative (i.e.~stable) eigenvalues since $c_0 > 0$ and thus is compatible with the spectral assumption \ref{assump2}.

Using that the heteroclinic connection is a subset of the (one-dimensional) unstable manifold of $(A^\ast,0,0,0,0)$, the unstable manifold of $(A^\ast,0,0,0,0)$ and the stable manifold of the origin intersect for $\gamma_1 = 0$ (and also $\varepsilon = 0$).
Furthermore, since the origin is a stable fixed point, its stable manifold is a 5-dimensional manifold in $\R^5$ and thus, the intersection has to be transversal.
Thus, the heteroclinic connection persists under small perturbations, in particular for small $(\varepsilon,\gamma_1)$.
By using the center manifold result Lemma \ref{lem:centerManifold}, we then obtain the existence of modulating traveling fronts.

\begin{theorem}\label{thm:S5}
	Let $c = 3c_u + \varepsilon c_0$, $c_v = -3c_u \neq -c_u$ and $\gamma_2 = 0$.
	Additionally, let $c_0^\ast > 0$ such that the assumptions \ref{assump1}--\ref{assump2} hold and set $c_0 > c_0^\ast$.
	Then, there exist $\varepsilon_0 > 0$ and $\gamma_1^\ast > 0$ such that for all $\varepsilon \in (0,\varepsilon_0)$ and $\gamma_1 \in (-\gamma_1^\ast, \gamma_1^\ast)$ the system \eqref{eq:SHDis}--\eqref{eq:ConDis} has a family of modulating traveling front solutions $(u_\text{front},v_\text{front})(\xi,p)$ satisfying \eqref{eq:modfrontBC} and
	\begin{align*}
		u_\text{front}(\xi,p) &= \varepsilon 2 \Re(\Ahet(\varepsilon^2\xi)) \cos(p) + \curlO(\varepsilon^2), \\
		v_\text{front}(\xi,p) &= \varepsilon^2 \left( \Bhet(\varepsilon^2\xi) + 2\Re\left(\dfrac{-2\gamma_1 + i \gamma_2}{2 - i (c_u+c_v)} \Ahet(\varepsilon^2\xi)^2\right)\cos(2p)\right) + \curlO(\varepsilon^3),
	\end{align*}
	where $(\Ahet, \tilde{A}_\text{het},\Bhet)$ is a heteroclinic orbit of \eqref{eq:redEqS5_1}--\eqref{eq:redEqS5_3} with
	\begin{align*}
		\lim_{\tilde{\xi} \rightarrow -\infty} (\Ahet, \tilde{A}_\text{het},\Bhet)(\tilde{\xi}) = (A^\ast,0,0) \text{ and } \lim_{\tilde{\xi} \rightarrow +\infty} (\Ahet, \tilde{A}_\text{het},\Bhet)(\tilde{\xi}) = (0,0,0).
	\end{align*}
	Here $\xi = x-ct$ and $p = x-\cphase t$ with $\cphase = c_u + \varepsilon^2 \omega_0^\ast$ given in Lemma \ref{lem:travelingWaves}.

\end{theorem}

\section*{Acknowlegements}
I would like to thank Bj\"orn de Rijk and Guido Schneider for their feedback and encouragement during the realization of this paper.
Funded by the Deutsche Forschungsgemeinschaft (DFG, German Research Foundation) under Germany's Excellence Strategy -- EXC 2075 -- 390740016. 
I acknowledge the support by the Stuttgart Center for Simulation Science (SimTech).

\bibliography{BibDeskLibraryModfrontsDispersion.bib}
\bibliographystyle{alphainitials}
\addcontentsline{toc}{chapter}{Bibliography}

\end{document}